\documentclass[12pt]{article}

\usepackage{amsmath,amsthm,amsfonts,stmaryrd}
\usepackage{amssymb,latexsym}
\usepackage{graphicx,color}
\usepackage{subfig}
\usepackage{setspace} 

\newtheorem{theorem}{Theorem}[section]
\newtheorem{prop}{Proposition}[section]
\newcommand{\set}[1]{\lbrace #1 \rbrace}

\newcommand{\jump}[1]{\llbracket #1 \rrbracket}
\newcommand{\mean}[1]{\{#1\}}

\newcommand{\abs}[1]{\lvert#1\rvert}
\newcommand{\norm}[1]{\lVert#1\rVert}

\newcommand{\A}[1]{A\Bigl( #1 \Bigr)}
\newcommand{\B}[1]{B\Bigl( #1 \Bigr)}
\newcommand{\D}[1]{D\Bigl( #1 \Bigr)}

\newcommand{\Dh}[1]{D_h^{\texttt{a}}\Bigl( #1 \Bigr)}

\newcommand{\Lh}[1]{L_h\Bigl( #1 \Bigr)}

\newcommand{\n}{\boldsymbol{n}}

\newcommand\bu{\boldsymbol{u}}
\newcommand\bv{\boldsymbol{v}}
\newcommand\bw{\boldsymbol{w}}

\newcommand\bx{\boldsymbol{x}}

\newcommand\br{\boldsymbol{r}}
\newcommand\bs{\boldsymbol{s}}

\newcommand\bn{\boldsymbol{n}}

\newcommand\bF{\boldsymbol{f}}
\newcommand\bI{\boldsymbol{I}}

\newcommand\bT{\boldsymbol{T}}

\newcommand\0{\mathbf{0}}
\newcommand\bFF{\boldsymbol{F}}
\newcommand\bB{\boldsymbol{B}}
\newcommand\bW{\boldsymbol{W}}

\newcommand{\cF}{\mathcal{F}}

\newcommand\cC{\mathcal{C}}

\newcommand\cP{\mathcal{P}}
\newcommand\cT{\mathcal{T}}

\newcommand\bcQ{\boldsymbol{\mathcal{Q}}}

\newcommand\bcW{\boldsymbol{\mathcal{W}}}

\newcommand\bcS{\boldsymbol{\mathcal{S}}}
\newcommand\bcN{\boldsymbol{\mathcal{N}}}

\newcommand\beps{\boldsymbol{\varepsilon}}
\newcommand\bsig{\boldsymbol{\sigma}}
\newcommand\btau{\boldsymbol{\tau}}

\newcommand\bphi{\boldsymbol{\varphi}}

\newcommand\bbC{\mathbb{C}}

\newcommand\R{\mathbb{R}}

\renewcommand\H{\mathrm{H}}
\renewcommand\L{\mathrm{L}}

\renewcommand\S{\Sigma}
\renewcommand\O{\Omega}
\newcommand\DO{\partial\O}

\newcommand\G{\Gamma}

\newcommand\bdiv{\mathop{\mathbf{div}}\nolimits}

\renewcommand\div{\mathop{\mathrm{div}}\nolimits}

\newcommand\tD{\mathtt{D}}
\newcommand\tr{\mathop{\mathrm{tr}}\nolimits}

\renewcommand\sp{\mathop{\mathrm{sp}}\nolimits}
\renewcommand\t{\mathtt{t}}

\newcommand\LO{\L^2(\O)}
\newcommand\HdivO{{\H(\mathbf{div},\O)}}
\newcommand\HsdivO{{\H^s(\mathbf{div},\O)}}

\newcommand\HsO{\H^s(\O)}

\newcommand\nxn{3\times 3}

\newcommand\sigmar{(\bsig, \br)}
\newcommand\taus{(\btau, \bs)}
\newcommand\sigmarh{(\bsig_h, \br_h)}
\newcommand\taush{(\btau_h, \bs_h)}
\newcommand\taushc{(\btau_h^c, \bs_h)}
\newcommand\tausht{(\tilde{\btau}_h, \mathbf{0})}

\newcommand\sigmarbar{(\bar\bsig, \bar\br)}

\renewcommand\sp{\mathop{\mathrm{sp}}\nolimits}
\renewcommand\t{\mathtt{t}}

\newcommand\disp{\displaystyle}
\newcommand\qin{\qquad\hbox{in }}
\newcommand\qon{\qquad\hbox{on }}

\date{}
\title{A mixed discontinuous Galerkin method for the time harmonic 
elasticity  problem with reduced symmetry 
\thanks{Partially supported by  
Spain's Ministry of Education Project MTM2013-43671-P and the Australian 
Research Council Grant DP120101886.}}

\author{
{\sc Antonio M\'arquez}\thanks{Departamento de Construcci\'on e Ingenier\'\i a de
 Fabricaci\'on, Universidad de Oviedo, Oviedo, Espa\~na,
 e-mail: {\tt amarquez@uniovi.es}}, $\,\,$
{\sc Salim Meddahi}\thanks{Departamento de Matem\'aticas, Facultad de Ciencias,
Universidad de Oviedo, Calvo Sotelo s/n, Oviedo, Espa\~na,
e-mail: {\tt salim@uniovi.es}}  \, and\,  
{\sc Thanh Tran}\thanks{School of Mathematics and Statistics,
University of New South Wales, Sydney NSW 2052, Australia, e-mail: {\tt
thanh.tran@unsw.edu.au}}
}

\begin{document}

\maketitle

\begin{abstract}
The aim of this paper is to analyze a mixed discontinuous Galerkin 
discretization of the time-harmonic elasticity problem.  The symmetry of the Cauchy stress 
tensor is imposed weakly, as in the traditional dual-mixed setting. 
We show that the discontinuous Galerkin scheme is well-posed and uniformly stable 
with respect to the mesh parameter $h$ and the Lam\'e coefficient $\lambda$.  
We also derive optimal a-priori error bounds in the energy norm.  Several numerical tests 
are presented in order to illustrate the performance of the method and confirm the theoretical results.
\end{abstract}
\medskip

\noindent
\textbf{Mathematics Subject Classification.} 65N30, 65N12, 65N15,  74B10
\medskip

\noindent
\textbf{Keywords.} 
Mixed elasticity equations,  finite elements, discontinuous Galerkin methods, 
error estimates.

\pagestyle{myheadings}
\thispagestyle{plain}
\markboth{A. M\'ARQUEZ, S. MEDDAHI, T. TRAN}{A  discontinuous Galerkin method for an  
elasticity  problem}

\section{Introduction}\label{section:1}

In this paper we are interested in the dual-mixed formulation of the  elasticity problem 
with weakly imposed symmetry. We introduce and analyze a mixed interior penalty discontinuous Galerkin  
(DG) method for the elasticity system in time-harmonic regime. The interior penalty DG method can be traced back 
to \cite{arnoldIP, douglas} and its application for elliptic  problems is now well 
understood; see \cite{DiPietroErn} and the references cited therein for more details. 
The mixed interior penalty method introduced here can be viewed as 
a discontinuous version of the Arnold-Falk-Winther div-conforming finite element space \cite{afw-2007}. 
It approximates the unknowns of the mixed 
formulation, given by the Cauchy stress tensor and the rotation, by discontinuous finite element spaces 
of degree $k$ and $k-1$ respectively.  This permits one to enjoy 
the well-known flexibility properties of DG methods for $hp$-adaptivity and to implement 
high-order elements by using standard shape functions. Moreover, our scheme is immune 
to the locking phenomenon that arises in the nearly incompressible 
case.

The first step in our study of the mixed DG scheme consists in providing a   
convergence analysis for the corresponding div-conforming Galerkin method based on  
the Arnold-Falk-Winther element.  We point out that there are many finite element methods for the mixed formulation 
of the elasticity problem with reduced symmetry  \cite{afw-2007, BoffiBrezziFortin, 
bernardo1, gopala,  st}. All of them have been analyzed in the static case, 
i.e., in the case $\omega=0$ in problem \eqref{model} below.  
In time harmonic regime, the operator 
underlying the mixed formulation is not Fredholm of index zero as in the classical displacement-based formulation. 
The same challenge is encountered when analyzing the curl-conforming variational formulation of the Maxwell 
system \cite{buffa, gatica2}. Actually, the abstract theory given in \cite{buffa} can also be applied to the 
dual-mixed variational formulation of  linear elasticity as shown (implicitly) in  
the analysis given in \cite{gatica1} for a fluid-solid interaction problem. Instead of using this approach,  
we take here advantage of the recent spectral analysis obtained in \cite{MMR} to directly 
deduce the stability of the Arnold-Falk-Winther finite element approximation 
of the indefinite elasticity problem.

An interior penalty discontinuous Galerkin method has also been introduced in \cite{dominik1} 
for the  Maxwell system. The DG formulation we are considering here is, in a certain sense,  its 
counterpart in the $\textrm{H}(\div)$-setting. Notice that, in contrast to \cite{dominik1}, our approach does not rely 
on a duality technique. We prove the convergence of the DG scheme by exploiting  the stability   
of the corresponding div-conforming  method and without requiring further 
regularity assumption than the one needed to write properly the right-hand side of \eqref{DGshort} below.   
Moreover, if the analytic Cauchy stress tensor, its divergence and rotation belong to a 
Sobolev space with  regularity exponent $s>1/2$, then  it is shown that the error in the DG-energy 
norm converges with the optimal order $O(h^{\min(s, k)} )$ with respect to the mesh size $h$ and 
the polynomial degree $k$.

The paper is organized as follows. In Section \ref{section:2}, we recall the 
dual formulation of the linear elasticity problem with reduced symmetry and prove 
its well-posedness when the wave number is different from a countable set of 
singular values. In Section \ref{section:3} we prove the convergence of the 
conforming Galerkin scheme based on the Arnold-Falk-Winther element. In  
Section \ref{section:4}, we introduce the mixed interior penalty discontinuous 
Galerkin method and its convergence analysis is carried out in Section \ref{section:5}.
Finally, in Section \ref{section:6} we present numerical results that confirm 
the theoretical convergence estimates.

We end this section with some of the notations that we will  use below. Given
any Hilbert space $V$, let $V^3$ and $V^{\nxn}$ denote, respectively,
the space of vectors and tensors of order $3$  with
entries in $V$. In particular, $\bI$ is the identity matrix of
$\R^{\nxn}$ and $\mathbf{0}$ denotes a generic null vector or tensor. 
Given $\btau:=(\tau_{ij})$ and $\bsig:=(\sigma_{ij})\in\R^{\nxn}$, 
we define as usual the transpose tensor $\btau^{\t}:=(\tau_{ji})$, 
the trace $\tr\btau:=\sum_{i=1}^3\tau_{ii}$, the deviatoric tensor 
$\btau^{\tD}:=\btau-\frac{1}{3}\left(\tr\btau\right)\bI$, and the
tensor inner product $\btau:\bsig:=\sum_{i,j=1}^3\tau_{ij}\sigma_{ij}$. 

Let $\O$ be a polyhedral Lipschitz bounded domain of $\R^3$ with
boundary $\DO$. For $s\geq 0$, $\norm{\cdot}_{s,\O}$ stands indistinctly
for the norm of the Hilbertian Sobolev spaces $\HsO$, $\HsO^3$ or
$\HsO^{\nxn}$, with the convention $\H^0(\O):=\LO$. We also define for
$s\geq 0$ the Hilbert space 
$\HsdivO:=\set{\btau\in\HsO^{\nxn}:\ \bdiv\btau\in\HsO^3}$, whose norm
is given by $\norm{\btau}^2_{\HsdivO}
:=\norm{\btau}_{s,\O}^2+\norm{\bdiv\btau}^2_{s,\O}$ and denote
$\HdivO:={\H^0(\mathbf{div};\O)}$.

Henceforth, we denote by $C$ generic constants independent of the discretization
parameter, which may take different values at different places.

\section{The model problem}\label{section:2}

Let $\O\subset \R^3$ be an open bounded Lipschitz 
polyhedron representing a solid domain. We denote by $\bn$ the outward unit normal 
vector to  $\DO$ and assume that  $\DO=\G\cup\S$ with $\textrm{int}(\G)\cap \mathrm{int}(\S) = \emptyset$.  
The solid is supposed to be isotropic
and linearly elastic with mass density $\rho$ and Lam\'e constants $\mu$
and $\lambda$.  Under the hypothesis of small oscillations, 
the time-harmonic elastodynamic  equations  with angular frequency $\omega>0$ 
and body force $\bF:\O \to \R^3$ are given by 
\begin{subequations}\label{model}
\begin{align}
\bsig=\cC\beps(\bu) &\qin\O,\label{model-a}
 \\
 \bdiv\bsig+\rho \omega^2\,\bu=\bF &\qin\O,\label{model-b}
 \\
 \bsig\bn=\mathbf{0} &\qon\S,\label{model-c}
 \\
 \bu=\mathbf{0} &\qon\G,\label{model-d}
\end{align}
\end{subequations}
where $\bu:\O \to \R^3$ is the displacement field, $\beps(\bu):=\frac{1}{2}\left[\nabla\bu+(\nabla\bu)^{\t}\right]$
is the linearized strain tensor  and $\cC$ is the elasticity operator defined by 
\[
\cC\btau:=\lambda\left(\tr\btau\right)\bI+2\mu\btau. 
\]

Our aim is to introduce the Cauchy stress tensor $\bsig:\O \to \R^{\nxn}$ as a  
primary variable in the variational formulation of \eqref{model}. To this end, we consider the closed subspace 
of  $\HdivO$ given by 
\[
 \bcW := \set{ \btau \in \HdivO;\quad \text{$\btau\bn = \0$ on $\S$} }
\]
and the space of skew symmetric tensors
\[
 \bcQ:= \set{\bs \in \LO^{\nxn}; \quad \bs = - \bs^{\t}}.
\]
Introducing  the rotation $\br := \frac{1}{2}\left[\nabla\bu - (\nabla\bu)^{\t}\right]$,  the constitutive equation 
\eqref{model-a} can be rewritten as, 
\[
 \cC^{-1}\bsig = \nabla \bu - \br. 
\]
Testing the last identity with $\btau\in \bcW$, integrating by parts and using  the momentum equation 
\eqref{model-b} to eliminate the displacement $\bu$, we end up with the following mixed variational formulation of problem 
\eqref{model}:
\begin{subequations}\label{varForm}
\text{find $\bsig\in \bcW$ and $\br\in \bcQ$ such that}
\begin{align}
\int_{\O} \bdiv \bsig \cdot \bdiv \btau - \kappa^2 \left( \int_{\O} \cC^{-1}\bsig:\btau + \int_{\O}\br:\btau\right) 
= \int_{\O} \bF\cdot \bdiv \btau & \quad \forall \btau\in \bcW\label{varForm-a}
\\
\int_{\O} \bsig:\bs = \0 & \quad \forall \bs \in \bcQ\label{varForm-b},
\end{align}
\end{subequations}
where the wave number $\kappa$  is given by $\sqrt{\rho}\,\omega$. 
We notice that equation \eqref{varForm-b} is a restriction that imposes weakly the symmetry of $\bsig$, and $\br$ 
is the corresponding Lagrange multiplier. We also point out that the dual formulation \eqref{varForm} 
degenerates as $\omega\to 0$. The static case $\omega=0$ is then not covered by our analysis.

We introduce the symmetric bilinear forms 
\[
 \B{\sigmar, \taus}:= \int_{\O} \cC^{-1}\bsig:\btau + \int_{\O}\br:\btau + \int_{\O}\bs:\bsig
\]
and
\[
 \A{\sigmar, \taus} := \int_{\O} \bdiv \bsig \cdot \bdiv \btau + \B{\sigmar, \taus}
\]
and denote the product norm on $\HdivO\times \L^2(\O)$ by
\[
 \norm{\taus}^2 := \norm{\btau}^2_{\HdivO} + \norm{\bs}^2_{0,\O}.
\]

\begin{prop}\label{infsupA-cont}
 There exists a constant $\alpha_A>0$, depending on $\mu$ and $\O$ (but not on $\lambda$), such that 
 \begin{equation}\label{infsupa}
  \sup_{\taus\in \bcW\times \bcQ} \frac{\A{\sigmar, \taus}}{\norm{\taus}} \geq \alpha_A \norm{\sigmar}\quad \forall \sigmar \in 
  \bcW\times \bcQ.
 \end{equation}
\end{prop}
\begin{proof}
It is important to notice that the bilinear form  
\begin{equation}
\label{invcCop}
\int_{\O}\cC^{-1}\bsig:\btau = \frac{1}{2\mu}\int_{\O}\bsig^{\tD}:\btau^{\tD} + \frac{1}{3(3\lambda + 2\mu)} \int_{\O}(\tr\bsig)(\tr\btau)
\end{equation}
is bounded by a constant independent of $\lambda$ when $\lambda$ is too large in comparison
with $\mu$. Moreover, it is shown in \cite[Lemma 2.1]{MMR} that there exists a 
constant $\alpha_0>0$, depending on $\mu$ and $\O$ (but not on $\lambda$), such that 
\begin{equation}\label{elip0}
\int_{\O} \cC^{-1}\btau:\btau + \int_{\O} \bdiv \btau \cdot \bdiv \btau 
\geq\alpha_0\norm{\btau}^2_{\HdivO} 
\qquad\forall\btau\in\bcW.
\end{equation}

On the other hand, there exists a constant $\beta>0$ depending only on $\O$ 
(see, for instance, \cite{BoffiBrezziFortin}) such that 
 \[
  \sup_{\btau\in \bcW} \frac{\int_{\O}\bs:\btau}{\norm{\btau}_{\HdivO}} \geq \beta \norm{\bs}_{0,\O}, \qquad 
\forall \bs\in \bcQ.  
 \]

The Babu\v{s}ka-Brezzi theory  
shows  that, for any bounded linear form $L\in \mathcal{L}(\bcW\times \bcQ)$, the 
problem: find $\sigmar\in \bcW\times \bcQ$ such that 
\[
 \A{\sigmar, \taus} = L\big(\taus\big)\qquad \forall \taus \in \bcW\times \bcQ,
\]
is well-posed, which proves  \eqref{infsupa}.
\end{proof}

We deduce from Proposition \ref{infsupA-cont} and the symmetry of 
$A(\cdot, \cdot)$ that the operator $\bT: \bcW\times \bcQ \to \bcW\times \bcQ$ 
characterized by 
\[
 \A{\bT\sigmar, \taus} = \B{\sigmar, \taus} \quad \forall \taus \in \bcW\times \bcQ
\]
is well-defined and bounded. It is clear that,  for a given wave number $\kappa>0$, $\sigmar\neq 0$ 
is a solution to the homogeneous version of problem \eqref{varForm} if and only if
$\left( \eta=\frac{1}{1+\kappa^2}, \sigmar\right)$ is an eigenpair for $\bT$. The following characterization 
of the spectrum of $\bT$  will  be useful for our analysis. 
\begin{prop}\label{specT}
 The spectrum $\sp(\bT)$ of $\bT$ decomposes as follows 
 \[
  \sp(\bT) = \set{0, 1} \cup \set{\eta_k}_{k\in \mathbb{N}}
 \]
where $\set{\eta_k}_k$ is a real sequence of  
finite-multiplicity eigenvalues of  $\bT$ which converges to 0. 
Moreover, $\eta=1$ is an infinite-multiplicity eigenvalue of $\bT$ while $\eta=0$ is not 
an eigenvalue. 
\end{prop}
\begin{proof}
See \cite[Theorem 3.7]{MMR}.
\end{proof}

\begin{theorem}\label{wellposed}
If  $\frac{1}{1+\kappa^2}\notin \sp(\bT)$, then \eqref{varForm}
is well-posed. Moreover, there exists a constant $C>0$ independent of $\lambda$ 
such that, for any $\bF\in \LO^3$, the solution $(\bsig, \br)\in \bcW\times \bcQ$ 
of \eqref{varForm} satisfies
\begin{equation}\label{estimsigr}
 \norm{\sigmar} \leq C \norm{\bF}_{0,\O}.
\end{equation}
\end{theorem}
\begin{proof}
Let us first recall that, given $z\in\bbC\setminus\set{\sp(\bT)}$, the resolvent  
 \[
  (z\bI-\bT)^{-1} :\bcW\times \bcQ\longrightarrow\bcW\times \bcQ  
 \]
is  bounded in $\mathcal{L}(\bcW\times \bcQ)$ by a constant $C$ only depending 
on $\O$ and $\abs{z}$.  

We deduce from \eqref{infsupa} and the symmetry of $A(\cdot, \cdot)$ 
that the problem: find $(\bar\bsig, \bar\br)\in \bcW\times \bcQ$ solution of 
\[
 \A{ \sigmarbar, \taus } = 
 \int_{\O} \bF\cdot \bdiv \btau,  \quad \forall \taus\in \bcW\times \bcQ,
\]
is well-posed. The solution of problem \eqref{varForm} is then given by 
\[
 \sigmar = \frac{1}{1+\kappa^2} \Big( \frac{\bI}{1+\kappa^2}  - \bT \Big)^{-1} \sigmarbar,
\]
and \eqref{estimsigr} follows from   the boundedness of 
$\Big( \frac{\bI}{1+\kappa^2}  - \bT \Big)^{-1}$.
\end{proof}

\section{A continuous Galerkin discretization}\label{section:3}

We consider  shape regular affine meshes $\mathcal{T}_h$ that subdivide the domain $\bar \Omega$ into  
tetrahedra $K$ of diameter $h_K$. The parameter $h:= \max_{K\in \cT_h} \{h_K\}$ 
represents the mesh size of $\cT_h$. Hereafter, given an integer $m\geq 0$ and a domain 
$D\subset \mathbb{R}^3$, $\cP_m(D)$ denotes the space of polynomials of degree at most $m$ on $D$.
The space of piecewise polynomial functions of degree at most $m$ relatively to $\cT_h$ is denoted by  
\[
 \cP_m(\cT_h) :=\set{ v\in L^2(\O); \quad v|_K \in \cP_m(K),\quad \forall K\in \cT_h }. 
\]

For any $k\geq 1$, we consider the finite element spaces
\[
 \bcW_h^c := \cP_k(\cT_h)^{\nxn} \cap \bcW 
 \qquad \text{and} \qquad \bcQ_h := \cP_{k-1}(\cT_h)^{\nxn} \cap \bcQ.
\]

Let us now recall some well-known properties of the Brezzi-Douglas-Marini (BDM) 
mixed finite element \cite{BDM}. Let $\hat K$ be a fixed reference tetrahedron. Given $K\in \cT_h$, 
there exists an affine and bijective map 
$\bFF_K:\, \hat K \to \R^3$ such that $\bFF_K(\hat K) = K$. We consider $\bB_K:= \nabla \bFF_K$ 
and define 
\[
 \bcN_{k}(K):= \set{\bw:\Omega \to \mathbb{R}^3;\quad \bw\circ \bFF_K = \bB_K^{-\t} \check{\bw}; \quad \check \bw \in \cP_{k-1}(\hat K)^3 \oplus 
 \bcS_k(\hat K)}
\]
where 
\[
 \bcS_k(\hat K) := \set{\check \bw \in \tilde{\cP}_k(\hat K)^3; \quad \check \bw \cdot \hat{\bx} = 0}
\]
with $\tilde{\cP}_k(\hat K)$ representing the space of homogeneous polynomials of total degree exactly $k$ in $\hat \bx\in 
\hat K$.

A polynomial $\bv\in \cP_k(K)^3$ is uniquely determined by the set of BDM degrees of freedom
\begin{align}
 m_\phi(\bv) &:= \int_F\bv\cdot \bn_K \phi \qquad\text{for all $\phi\in \cP_k(F)$, for all $F\in \cF(K)$}\label{DFa}
 \\
 m_{\bw}(\bv) &:=  \int_K \bv \cdot \bw \qquad\text{for all $\bw\in \bcN_{k-1}(K)$}\label{DFb},
\end{align} 
where $\bn_K$ is the outward unit normal vector to $\partial K$. 
Conditions \eqref{DFb} are avoided in the case $k=1$.

Let us consider an arbitrary, but fixed, orientation of all internal faces $F$ of $\cT_h$   
by  normal vectors $\bn_F$. On the faces $F$ lying on $\partial \O$ we take 
$\bn_F = \bn|_{F}$. We can introduce the global BDM-interpolation operator 
$\Pi_h: \H(\textrm{div}, \O)\cap \H^s(\O)^3 \to \bcW_h^c$,  
characterized, for any $\bv\in \H(\textrm{div}, \O)\cap \H^s(\O)^3$ with $s>1/2$, by the conditions
 \begin{align}\label{Pih}
\int_F \Pi_h\bv\cdot \bn_F \phi & =  \int_F \bv\cdot \bn_F \phi 
\quad&\text{for all $\phi\in \cP_k(F)$, for all $F\in \cF_h$},
 \\
\int_K \Pi_h\bv \cdot \bw &= \int_K \bv \cdot \bw \quad&\text{for all $\bw\in \bcN_{k-1}(K)$, for all $K\in \cT_h$}\label{DFGc}.
\end{align}
We have the following classical error estimate, see \cite{BoffiBrezziFortin},
\begin{equation}\label{asymp}
 \norm{\bv - \Pi_h \bv}_{0,\O} \leq C h^{\min(s, k+1)} \norm{\bv}_{s,\O} \qquad \forall \bv \in \HsO^{3}, \text{with $s>1/2$}.
\end{equation}
Moreover, thanks to the commutativity property, if $\div \bv \in \HsO$, then
\begin{equation}\label{asympDiv}
 \norm{\div (\bv - \Pi_h \bv) }_{0,\O} = \norm{\div \bv - \mathcal R_h \div \bv) }_{0,\O} 
 \leq C h^{\min(s, k)} \norm{\div\bv}_{s,\O},
\end{equation}
where $\mathcal R_h$ is the $\LO$-orthogonal projection onto $\cP_{k-1}(\cT_h)$. Finally, 
we denote by $\mathcal S_h:\ \bcQ\to\bcQ_h$ the orthogonal
projector with respect to the $\LO^{\nxn}$-norm. 
It is well-known that, for any $s\in(0,1]$, we have
\begin{equation}\label{asymQ}
 \norm{\br-\mathcal S_h\br}_{0,\O}
\leq C h^{\min(s, k)} \norm{\br}_{s,\O}
 \qquad\forall\br\in\HsO^{\nxn}\cap\bcQ.
\end{equation}

We propose the following continuous Galerkin (CG) discretization of problem \eqref{varForm}: 
find $\bsig_h\in \bcW^c_h$ and $\br_h\in \bcQ_h$ such that
\begin{subequations}\label{CGForm}
\begin{align}
\int_{\O} \bdiv \bsig_h \cdot \bdiv \btau - \kappa^2 \left( \int_{\O} \cC^{-1}\bsig_h:\btau + \int_{\O}\br_h:\btau\right)
= \int_{\O} \bF\cdot \bdiv \btau & \quad \forall \btau\in \bcW_h^c\label{CGForm-a}
\\
\int_{\O} \bsig_h:\bs = \0 & \quad \forall \bs \in \bcQ_h\label{CGForm-b}.
\end{align}
\end{subequations}

\begin{prop}\label{infsupA-disc}
 There exists a constant $\alpha_A^{c}>0$ independent of $h$ and $\lambda$  such that 
 \begin{equation}\label{infsupA-disceq}
  \sup_{\taush\in \bcW_h^c\times \bcQ_h} \frac{\A{\sigmarh, \taush}}{\norm{\taush}} 
  \geq \alpha_A^{c} \norm{\sigmarh}\quad \forall \sigmarh \in 
  \bcW_h^c\times \bcQ_h.
 \end{equation}
\end{prop}
\begin{proof} 
We prove this result by following the same steps given in Proposition \ref{infsupA-cont}. 
We deduce from \eqref{elip0} that the bilinear form 
$
 \int_{\O} \bdiv \bsig \cdot \bdiv \btau + \int_{\O} \cC^{-1}\bsig:\btau
$
is elliptic on $\bcW_h^c$. 
Moreover, the following discrete inf-sup
condition is proved in \cite{afw-acta-2006, BoffiBrezziFortin}: There exists
$\beta^c>0$, independent of $h$, such that
\[
  \sup_{\btau_h\in \bcW_h^c} \frac{\disp\int_{\O}\bs_h:\btau_h}{\norm{\btau_h}_{\HdivO}} \geq \beta^c \norm{\bs_h}_{0,\O}.
 \]
Therefore, we can use the Babu\v{s}ka-Brezzi theory to 
ensure that, for any bounded linear form $L\in \mathcal{L}( \bcW\times \bcQ)$, the problem: find $\sigmarh\in \bcW_h^c\times \bcQ_h$
such that 
\[
 \A{\sigmarh, \taush} = L\big(\taush \big)\qquad \forall \taush \in \bcW_h^c\times \bcQ_h
\]
admits a unique solution and there exists a constant $C>0$ independent of $h$ and $\lambda$  such that 
\[
 \norm{\sigmarh} \leq C \norm{L}_{\mathcal{L}( \bcW\times \bcQ)},
\]
which gives \eqref{infsupA-disceq}.
\end{proof}

We can now consider the discrete counterpart $\bT_h:\, \bcW^c_h\times \bcQ_h \to \bcW^c_h\times \bcQ_h$ of $\bT$ 
characterized, for any $\sigmarh\in \bcW^c_h\times \bcQ_h$,  by 
\[
 \A{\bT_h\sigmarh, \taush} = \B{\sigmarh, \taush}\quad \forall \taush \in \bcW^c_h\times \bcQ_h.
\]
As a consequence of Proposition \ref{infsupA-disc}, $\bT_h$ is  well-defined and uniformly bounded with respect to 
$h$ and $\lambda$. Moreover, we  deduce  from \cite[Theorem 5.2]{MMR} that, 
if  $\frac{1}{1+\kappa^2}\notin\sp(\bT)$,  
there exists a mesh size $h_0>0$  such that, for $h\leq h_0$,  
\begin{equation}\label{boundTh}
 \norm{\big(\frac{\bI}{1+\kappa^2}   - \bT_h\big)\taush}  \geq C_0 \norm{\taush} \quad 
 \forall \taush\in \bcW^c_h\times \bcQ_h,
\end{equation}
with a constant  $C_0>0$ independent of $h$ and $\lambda$.

We introduce the bilinear form 
\[
   \D{\sigmarh, \taush}:=\A{\sigmarh, \taush} - (1+\kappa^2) \B{\sigmarh, \taush}
\]
and notice that there exists a constant $M^c_D>0$ independent of $h$ and $\lambda$  such that 
\begin{equation}\label{boundD}
 \Big| \D{\sigmar, \taus} \Big| \leq M^c_D  \,
 \norm{\sigmar}\,  \norm{\taus} \qquad \forall \sigmar, \, \taus \in  \bcW\times \bcQ.
\end{equation}

\begin{prop}\label{infsupD}
Assume that  $\frac{1}{1+\kappa^2}\notin\sp(\bT)$ and let $h_0>0$ be the parameter for which 
\eqref{boundTh} holds true for all $h\leq h_0$. Then, for  $h\leq h_0$, 
\begin{equation}\label{infsupABh0}
  \inf_{\taush\in \bcW_h^c\times \bcQ_h} \frac{\D{\sigmarh, \taush}}{\norm{\taush}} 
  \geq \alpha_D^c  \norm{\sigmarh},\qquad \forall \sigmarh \in 
  \bcW_h^c\times \bcQ_h,
\end{equation}
with $\alpha_D^c>0$ independent of the mesh size $h$ and $\lambda$. 
\end{prop}
\begin{proof}
 We deduce from Proposition \ref{infsupA-disc} that there exists an 
 operator $\Theta_h:\, \bcW^c_h\times \bcQ_h \to \bcW^c_h\times \bcQ_h$ satisfying
 \begin{equation}\label{cota1}
  \A{\taush, \Theta_h \taush} = \norm{\taush}^2 \quad \text{and}\quad
  \norm{\Theta_h \taush} \leq \frac{1}{\alpha_A^c} \norm{\taush}
 \end{equation}
 for all $\taush\in \bcW^c_h\times \bcQ_h$.
It follows from \eqref{boundTh} and \eqref{cota1} that 
\begin{multline*}
 \D{ \taush, \Theta_h \big(\frac{\bI}{1+\kappa^2} - \bT_h \big) \taush } \\= (1+\kappa^2) 
 \A{ \big(\frac{\bI}{1+\kappa^2} - \bT_h \big) \taush, \Theta_h \big(\frac{\bI}{1+\kappa^2} - \bT_h \big) \taush}\\
 = (1+\kappa^2)    \norm{\big(\frac{\bI}{1+\kappa^2} - \bT_h \big) \taush}^2 \ge (1+\kappa^2)C_0  \norm{\taush} 
 \, \norm{\big(\frac{\bI}{1+\kappa^2} - \bT_h \big) \taush}
\end{multline*}
for all $\sigmarh \in 
  \bcW_h^c\times \bcQ_h$, with the constant $C_0>0$ from \eqref{boundTh}. 
  The result follows now with $\alpha_D^c = C_0 (1+\kappa^2)$.
\end{proof}

\begin{theorem}\label{ConvCG}
Assume that  $\frac{1}{1+\kappa^2}\notin\sp(\bT)$ and let $h_0>0$ be the parameter for which 
\eqref{boundTh} holds true for all $h\leq h_0$. Then, for  $h\leq h_0$,    
we have the following C\'ea estimate,
\begin{equation}\label{Cea}
 \norm{\sigmar - \sigmarh} \leq \left( 1 + \frac{M^c_D}{\alpha_D^c} \right) \inf_{\taush\in \bcW_h^c\times \bcQ_h} 
 \norm{\sigmar - \taush}, \quad \forall h\leq h_0.
\end{equation}
Moreover, if the exact solution $\bu$ of \eqref{model} belongs to $\H^{1+s}(\O)^3$ and 
$\bdiv \bsig \in \H^{s}(\O)^3$ for some $s>1/2$ then,   
\[
 \norm{\sigmar - \sigmarh} \leq \, C\,  h^{\min(s, k)}\,  ( \norm{\bu}_{1+s,\O} + 
 \norm{\bdiv\bsig}_{s,\O}), \quad \forall h\leq h_0,
\]
with $C>0$ independent of $h$ and $\lambda$.
\end{theorem}
\begin{proof}
The C\'ea estimate \eqref{Cea} is a direct consequence of  
\eqref{boundD} and \eqref{infsupABh0}. The asymptotic error estimate follows from \eqref{asymp}, \eqref{asympDiv} and \eqref{asymQ}.
\end{proof}

\section{A discontinuous Galerkin discretization}\label{section:4}

From now on we assume that there exists $s_0>1/2$ such that 
$\bF|_{\O_j}\in \H^{s_0}(\O_j)$ for $j= 1,\cdots, J$, where 
$\set{\O_j,\quad j=1\cdots, J}$ is a set of polyhedral subdomains
forming a disjoint partition of $\bar \O$, i.e., 
\[
 \O_j\cap\O_i = \emptyset \quad \text{for all $1\leq i\neq j\leq J$ \quad and \quad $\bar\O = \cup_{j= 1}^J \bar{\O}_j$}.
\] 
We deduce from this additional hypothesis on $\bF$ and 
the momentum equation \eqref{model-b} that $(\bdiv \bsig)|_{\O_j}$ belongs to 
$\H^{\min(s_0, 1)}(\O_j)$ for any  $j= 1,\cdots, J$.

In what follows, we assume that $\mathcal{T}_h$ is compatible with the partition $\bar\O = \cup_{j= 1}^J \bar{\O}_j$, i.e., 
\[
 \disp\cup_{K\in \cT_h(\O_j)} K= \bar{\O}_j \quad \forall j=1,\cdots, J,
\]
where $\cT_h(\O_j):= \set{K\in\cT_h,\quad  K\subset \bar{\O}_j}$.

We say that a closed subset $F\subset \overline{\Omega}$ is an interior face if $F$ has a positive 2-dimensional 
measure and if there are distinct elements $K$ and $K'$ such that $F = K\cap K'$. A closed 
subset $F\subset \overline{\Omega}$ is a boundary face if
there exists $K\in \cT_h$ such that $F$ is a face of $K$ and $F = K\cap \Gamma$. 
We consider the set $\cF_h^0$ of interior faces and the set $\cF_h^\partial$ of boundary faces.
We assume that the boundary mesh $\cF_h^\partial$ is compatible with the partition $\DO = \G \cup \S$, i.e., 
\[
\cup_{F\in \cF_h^\G} F = \G \qquad \text{and} \qquad \cup_{F\in \cF_h^\S} F = \S
\]
where $\cF_h^\G:= \set{F\in \cF_h^\partial; \quad F\subset \G}$ and 
$\cF_h^{\S}:= \set{F\in \cF_h^\partial; \quad F\subset \S}$.
We denote   
\[
  \cF_h := \cF_h^0\cup \cF_h^\partial\qquad \text{and} \qquad \cF^*_h:= \cF_h^{0} \cup \cF_h^{\S},
\]
and for any element $K\in \cT_h$, we introduce the set 
 \[
 \cF(K):= \set{F\in \cF_h;\quad F\subset \partial K} 
 \]
 of faces composing the boundary of $K$.

For any $s\geq 0$, we consider the broken Sobolev space 
\[
 \H^s(\cT_h):=
 \set{\bv \in \L^2(\O)^3; \quad \bv|_K\in \H^s(\O)^3\quad \forall K\in \cT_h}.
\]
For each $\bv:=\set{\bv_K}\in \H^s(\cT_h)^3$ and 
$\btau:= \set{\btau_K}\in \H^s(\cT_h)^{\nxn}$
the components $\bv_K$ and $\btau_K$  represent the restrictions $\bv|_K$ and $\btau|_K$. 
When no confusion arises, the restrictions of these functions will be written 
without any subscript. We will also need the space given on the skeletons of the triangulations $\cT_h$  by 
\[
 \L^2(\cF_h):= \prod_{F\in \mathcal{F}_h} \L^2(F).
 \]
Similarly, the components $\mu_F$ 
of $\mu := \set{\mu_F}\in \L^2(\cF_h)$  
coincide with the restrictions $\mu|_F$ and we denote
\[
 \int_{\cF_h} \mu := \sum_{F\in \cF_h} \int_F \mu_F\quad \text{and}\quad 
 \norm{\mu}^2_{0,\cF_h}:= \int_{\cF_h} \mu^2, 
 \qquad 
 \forall \mu\in \L^2(\cF_h).
\]
From now on, $h_\cF\in \L^2(\cF_h)$ is the piecewise constant function 
defined by $h_\cF|_F := h_F$ for all $F \in \cF_h$ with $h_F$ denoting the 
diameter of face $F$.

Given a vector valued function $\bv\in \H^t(\cT_h)^3$, with $t>1/2$, 
we define averages $\mean{\bv}\in \L^2(\cF_h)^3$ and jumps $\jump{\bv}\in \L^2(\cF_h)$ 
by
\[
 \mean{\bv}_F := (\bv_K + \bv_{K'})/2 \quad \text{and} \quad \jump{\bv}_F := \bv_K \cdot\n_K + \bv_{K'}\cdot\n_{K'} 
 \quad \forall F \in \cF(K)\cap \cF(K'),
\]
where $\n_K$ is the outward unit normal vector to $\partial K$. On the boundary of $\O$ we use the following 
conventions for averages and jumps:
\[
 \mean{\bv}_F := \bv_K  \quad \text{and} \quad \jump{\bv}_F := \bv_K \cdot\n 
 \quad \forall F \in \cF(K)\cap \DO.
\]

Similarly, for matrix valued functions $\btau\in \H^t(\cT_h)^{\nxn}$, we define $\mean{\btau}\in \L^2(\cF_h)^{\nxn}$ and 
$\jump{\btau}\in \L^2(\cF_h)^3$ by
\[
 \mean{\btau}_F := (\btau_K + \btau_{K'})/2 \quad \text{and} \quad \jump{\btau}_F := 
 \btau_K \n_K + \btau_{K'}\n_{K'} 
 \quad \forall F \in \cF(K)\cap \cF(K')
\]
and on the boundary of $\Omega$ we set 
\[
 \mean{\btau}_F := \btau_K  \quad \text{and} \quad \jump{\btau}_F := 
 \btau_K \n  
 \quad \forall F \in \cF(K)\cap \DO.
\]

For any $k\geq 1$ we introduce the finite dimensional space 
$
 \bcW_h := \cP_k(\cT_h)^{\nxn}
$
and consider 
$
 \bcW(h) := \bcW + \bcW_h. 
$
Given $\btau \in \bcW_h$ we define $\bdiv_h \btau \in \L^2(\O)^3$ by 
$
\bdiv_h \btau|_{K} = \bdiv (\btau|_K)$ for all $K\in \cT_h
$ and 
endow $\bcW(h)$ with the seminorm
\[
 |\btau|^2_{\bcW(h)} := \norm{\bdiv_h \btau}^2_{0,\O} + \norm{h_{\cF}^{-1/2} \jump{\btau}}^2_{0,\cF^*_h}
\]
and the norm
\[
 \norm{\btau}^2_{\bcW(h)} := |\btau|^2_{\bcW(h)} + \norm{\btau}^2_{0,\O}.
\]
For the sake of simplicity, we will also use the notation 
\[
 \norm{(\btau, \bs)}^2_{DG} : =  \norm{\btau}^2_{\bcW(h)} + \norm{\bs}^2_{0,\O}.
\]

Given a parameter $\texttt{a}>0$, we introduce  the symmetric bilinear form 
\begin{multline*}
 \Dh{\sigmar,\taus}:= 
 \int_{\O} \bdiv_h \bsig \cdot \bdiv_h \btau - \kappa^2 \B{\sigmar, \taus} + \\
 \int_{\cF^*_h} \texttt{a}h_{\cF}^{-1}\, \jump{\bsig}\cdot \jump{\btau}-\int_{\cF^*_h} 
\left( \mean{\bdiv_h\bsig} \cdot \jump{\btau} 
+ \mean{\bdiv_h\btau} \cdot \jump{\bsig} \right)\quad \forall \sigmar, \taus \in \bcW_h\times \bcQ_h
\end{multline*}
and the linear form 
\[
 \Lh{\taus}:= \int_{\O} \bF\cdot \bdiv_h \btau- \int_{\cF^*_h} \mean{\bF}\cdot \jump{\btau}\qquad 
 \forall  \taus \in \bcW_h\times \bcQ_h,
\]
and consider the DG method: find $\sigmarh\in \bcW_h\times \bcQ_h$ such that 
\begin{equation}\label{DGshort}
\Dh{\sigmarh, \taus} = \Lh{\taus}\qquad \forall \taus\in \bcW_h\times \bcQ_h.
\end{equation}
We notice that, as it is usually the case for DG methods, the essential boundary condition is directly incorporated 
within the scheme. We need the following technical result to show that the bilinear form $D_h^{\texttt{a}}(\cdot, \cdot)$ is 
uniformly  bounded on $\bcW_h$. 
\begin{prop}\label{card} 
There exists a constant $C>0$ independent of $h$ such that 
 \begin{equation}\label{discTrace}
  \norm{h^{1/2}_{\cF}\mean{v}}_{0,\cF_h}\leq C \norm{v}_{0,\O}\quad \forall  v\in \cP_k(\cT_h). 
 \end{equation}
\end{prop}
\begin{proof}
It is straightforward that
 \[
  \norm{h^{1/2}_{\cF}\mean{v}}^2_{0,\cF_h} = \sum_{F\in \cF_h} \norm{h_F^{1/2}\mean{v}}^2_{0,F} \leq   
   \sum_{K\in \cT_h} \sum_{F\in \cF(K)} h_F\norm{v}^2_{0,F} \leq  
  \sum_{K\in \cT_h}  h_K\norm{v}^2_{0,\partial K}.
\]
The result follows now from the following discrete trace inequality (cf. \cite{DiPietroErn}):
 \begin{equation*}\label{discreteTrace}
  h_K \norm{v}^2_{0,\partial K} \leq C_0 \norm{v}^2_{0,K} \quad \forall v\in \cP_k(K), \quad \forall K\in \cT_h,
 \end{equation*}
 where $C_0>0$ is independent of $K$. 
\end{proof}

With the aid of the Cauchy-Schwarz inequality and Proposition \ref{card}, we can easily prove that there exists 
constants $M^{d}_D>0$  independent of $h$ and $\lambda$ such that 
\begin{equation}
 \label{boundDh}
 |\Dh{\sigmar, \taush}| \leq M^{d}_D  \Big(\norm{\bsig}^2_{\bcW(h)} + 
 \norm{h_{\cF}^{1/2} \mean{\bdiv \bsig}}^2_{0,\cF^*_h} + \norm{\br}^2_{0,\O}\Big)^{1/2} \norm{(\btau_h, \bs_h)}_{DG}
\end{equation}
for all $\sigmar\in\bcW(h)\times \bcQ$ with $\bdiv_h\bsig \in \H^s(\cT_h)^3$ for a given $s>1/2$ and for all $\taush \in \bcW_h\times \bcQ_h$.
\medskip

We end this section by showing that the DG scheme \eqref{DGshort} is consistent. 

\begin{prop}\label{consistency}
 Let $\bu$ be the solution of \eqref{model} and let $\bsig:= \cC \beps(\bu)$ and $\br:= \frac12(\nabla \bu - (\nabla \bu)^\t)$. 
 Then,
 \begin{equation}\label{consistent}
  D_h^{\textup{\texttt{a}}}\Big( (\bsig-\bsig_h, \br - \br_h), \taush\Big) =  0 \quad \forall \taush\in \bcW_h\times \bcQ_h.
 \end{equation}
\end{prop}
\begin{proof}
 By definition, 
 \begin{multline}
 \label{id1}
  \Dh{\sigmar, \taush} = \int_{\O} \bdiv\bsig \cdot \bdiv_h \btau_h - 
  \kappa^2 \left( \int_{\O} \cC^{-1}\bsig:\btau_h + \int_{\O}\br:\btau_h\right)\\ -
\int_{\cF^*_h}  \mean{\bdiv\bsig} \cdot \jump{\btau_h}.
 \end{multline}
The identity $\bdiv \bsig = \bF - \kappa^2 \bu$ and  integration by parts  yield
\begin{multline*}
 \int_{\O} \bdiv\bsig \cdot \bdiv_h \btau_h = \int_{\O} \bF \cdot \bdiv_h \btau_h - 
 \kappa^2 \sum_{K\in \cT_h} \int_K \bu \cdot \bdiv \btau_h = \int_{\O} \bF \cdot \bdiv_h \btau_h +\\ 
 \kappa^2\sum_{K\in \cT_h} \int_K \nabla \bu: \btau_h - \kappa^2\sum_{K\in \cT_h} \int_{\partial K} \bu\cdot \btau_h\bn_K.
\end{multline*}
Substituting back  into \eqref{id1} by taking into account that $\nabla \bu=\cC^{-1}\bsig - \br$ and 
\[
 \sum_{K\in \cT_h} \int_{\partial K} \bu\cdot \btau_h\bn_K = \int_{\cF^*_h} \mean{\bu}\cdot \jump{\btau_h}
\]
we obtain 
\begin{multline*}
\Dh{\sigmar, \taush} = \int_{\O} \bF \cdot \bdiv_h \btau_h - 
\int_{\cF^*_h}  \mean{\bdiv\bsig +\kappa^2\bu} \cdot \jump{\btau_h} \\= 
\int_{\O} \bF \cdot \bdiv_h \btau_h - 
\int_{\cF^*_h}  \mean{\bF} \cdot \jump{\btau_h}
\end{multline*}
and the result follows.
\end{proof}

\section{Well-posedness and stability of the DG method}\label{section:5}
By using  the transformation rules  
\begin{equation}\label{trans}
 \phi\circ \bFF_K = \hat \phi, \qquad \bv\circ \bFF_K = \disp\frac{\bB_K \hat \bv}{|\det \bB_K|}
 \qquad \text{and} \qquad \bw\circ \bFF_K = \bB_K^{-\t} \check{\bw},
\end{equation}
we can easily show that  
\begin{equation}\label{ident}
 \int_F\bv\cdot \bn_K \phi = \int_{\hat F} \hat\bv\cdot \bn_{\hat K} \hat \phi 
 \qquad \text{and} \qquad 
 \int_K \bv \cdot \bw = \int_{\hat K} \hat \bv \cdot \check \bw,
\end{equation}
where $F$ is the image of the face $\hat F$ under the affine map $\bFF_K:\, \hat K \to \R^3$ defined in Section 
\ref{section:3}. 

\begin{prop}\label{equivA}
There exists a constant $C>0$ independent of $h$ such that  
\begin{multline}
\left( \norm{\mathrm{div}\, \bv}^2_{0,K} + h_K^{-2} \norm{\bv}^2_{0,K}\right)^{1/2} 
  \leq C\Big(  h_K^{-1} \sup_{\bw\in \bcN_{k-1}(K)} \frac{\int_K \bv \cdot \bw}{\norm{\bw}_{0,K}}\\+
  \sum_{F\in \cF(K)} h_F^{-1/2} \sup_{\phi\in \cP_k(F)} \frac{\int_F\bv\cdot \bn_K \phi}{\norm{\phi}_{0,F}}\Big) 
\end{multline} 
for all $\bv \in \cP_k(K)^3$.
\end{prop}

\begin{proof}
We will use here the notation $a\lesssim b$ to express that there exists $C>0$ independent of $h$ 
such that $a\le C\,b$ for all $h$. The notation $A \simeq B$ means that $A\lesssim B$ and $B\lesssim A$ 
simultaneously. We first notice that, thanks to the unisolvency of conditions \eqref{DFa}-\eqref{DFb},  the norms 
\[
\left( \norm{\mathrm{div}\, \hat \bv}^2_{0,\hat K} +  \norm{\hat \bv}^2_{0, \hat K}\right)^{1/2} 
\quad 
\text{and} 
\quad
\sup_{\check \bw\in \bcN_{k-1}(\hat K)} \frac{\disp\int_{\hat K} \hat \bv \cdot \check \bw}{\norm{\check \bw}_{0,\hat K}} +
 \sum_{\hat F\in \cF(\hat K)}  \sup_{\hat \phi\in \cP_k(\hat F)} \frac{\disp\int_{\hat F}\hat \bv\cdot \bn_{\hat K} \hat \phi}
 {\norm{\hat \phi}_{0,\hat F}}
\]
are equivalent on the finite dimensional space $\cP_k(\hat K)^3$. Standard scaling arguments show that
\[
 h_K \norm{\bv}_{0,K}^2 \simeq  
 \norm{\hat \bv}^2_{0,\hat K},
 \qquad 
 h_K^{3} \norm{\text{div} \bv}_{0,K}^2 \simeq \norm{\mathrm{div}\, \hat \bv}^2_{0, \hat K}
\]
and
\[
 \norm{\phi}^2_{0,F} \simeq h_F^2 \norm{\hat \phi}^2_{0,\hat F}, \qquad 
 \norm{\bw}^2_{0,K} \simeq h_K \norm{\check \bw}^2_{0,\hat K}.
\] 
Hence, we deduce from \eqref{ident}  that 
\begin{multline*}
 \left( h_K \norm{\bv}^2_{0,K} + h_K^3 \norm{\text{div} \bv}_{0,K}^2   \right)^{1/2} \lesssim 
 \left( \norm{\hat \bv}^2_{0,\hat K} + \norm{\text{div} \hat \bv}^2_{0, \hat K} \right)^{1/2} \lesssim\\
 \sup_{\check \bw\in \bcN_{k-1}(\hat K)} \frac{\int_{\hat K} \hat \bv \cdot \check \bw}{\norm{\check \bw}_{0,\hat K}} +
 \sum_{\hat F\in \cF(\hat K)}  \sup_{\hat \phi\in \cP_k(\hat F)} \frac{\int_{\hat F}\hat \bv\cdot \bn_{\hat K} \hat \phi}
 {\norm{\hat \phi}_{0,\hat F}}  \lesssim\\
 h_K^{1/2} \sup_{\bw\in \bcN_{k-1}(K)} \frac{\int_K \bv \cdot \bw}{\norm{\bw}_{0,K}}+
  \sum_{F\in \cF(K)} h_F \sup_{\phi\in \cP_k(F)} \frac{\int_F\bv\cdot \bn_K \phi}{\norm{\phi}_{0,F}},
\end{multline*}
and the result follows.
\end{proof}

We introduce the projection $\cP_h:\, \bcW_h \to \bcW_h^c$ 
uniquely characterized, for any $\btau\in \bcW_h$, by the conditions
\begin{eqnarray}
 \int_F (\cP_h\btau) \bn_F \cdot \bphi & = & \int_F \btau \bn_F \cdot \bphi 
\qquad \forall \bphi\in \cP_k(F)^3, \quad\forall F\in \cF_h^{\G},\label{GDFb}
 \\
 \int_F (\cP_h\btau) \bn_F \cdot \bphi & = & 0 
\qquad \forall\bphi\in \cP_k(F)^3, \quad \forall F\in \cF_h^{\S},\label{GDFa}
 \\
 \int_F (\cP_h\btau) \bn_F \cdot \bphi & = & \int_F \mean{\btau}_F \bn_F \cdot \bphi
\qquad \forall \bphi\in \cP_k(F)^3, \quad\forall F\in \cF_h^{0},\label{GDFbc}
 \\
\int_K \cP_h\btau : \bW &= & \int_K \btau : \bW \quad\forall \text{$\bW$ with rows in $\bcN_{k-1}(K)$}, 
\quad \forall  K\in \cT_h\label{DFGd}.
\end{eqnarray}
We point out that the projection $\cP_h$ may be viewed as  the  div-conforming  
counterpart of the projection with curl-conforming range introduced in \cite{dominik1}.

\begin{prop}\label{propC}
The norm equivalence 
\begin{equation}\label{equivN}
\underbar{C}\,  \norm{\btau}_{\bcW(h)}  \leq \Big( \norm{\cP_h \btau}^2_{\HdivO} + 
 \norm{h_{\cF}^{-1/2} \jump{\btau}}^2_{0,\cF^*_h}
 \Big)^{1/2} \leq \bar{C} \norm{\btau}_{\bcW(h)} 
\end{equation}
holds true on $\bcW_h$ with constants $\underbar{C}>0$ and $\bar{C}>0$ independent of $h$.

\end{prop}
\begin{proof}
 Using Proposition \ref{equivA} row-wise we deduce that there exists $C_0>0$ independent of $h$ such that 
 \begin{multline*}
 \norm{\bdiv (\btau- \cP_h \btau)}^2_{0,K} + h_K^{-2} \norm{\btau- \cP_h \btau}^2_{0,K} \\
  \leq C_0   
  \sum_{F\in \cF(K)} h_F^{-1} \Big( \sup_{\bphi\in \cP_k(F)^3} \frac{\int_F(\btau- \cP_h \btau)\bn_K\cdot \bphi}{\norm{\bphi}_{0,F}}\Big)^2. 
\end{multline*}
It is easy to obtain, from the definition of $\cP_h$, the identity
\[
 \int_F(\btau- \cP_h \btau)\bn_K\cdot \bphi = \begin{cases}
                                               \frac{1}{2}\int_F \jump{\btau}_F\cdot \bphi & \text{if $F\in \cF^0_h$}\\[1ex]
                                               \int_F \jump{\btau}_F\cdot \bphi & \text{if $F\in \cF^\S_h$}\\[1ex]
                                               0 & \text{if $F\in \cF^\G_h$}.
                                              \end{cases}
\]
Hence, using the Cauchy-Schwarz inequality and 
summing up over $K\in \cT_h$ we deduce that 
\begin{equation}\label{proj}
 \norm{\bdiv_h (\btau- \cP_h \btau)}^2_{0,\O} + \sum_{K\in \cT_h} h_K^{-2} \norm{\btau- \cP_h \btau}^2_{0,K}
 \leq C_0  
 \sum_{F\in \cF_h^*} h_F^{-1} \norm{\jump{\btau}}^2_{0,F},
\end{equation}
which proves that 
\begin{equation}\label{propB}
 \norm{\btau - \cP_h \btau}^2_{\bcW(h)} \leq (1 + C_0) \norm{h_{\cF}^{-1/2} \jump{\btau}}^2_{0,\cF^*_h}
 \qquad \forall \btau\in \bcW_h.
\end{equation}
The lower bound of \eqref{equivN} is then a consequence of the uniform boundedness of 
$\cP_h$ on $\bcW_h$,
\begin{equation}\label{PhBounded}
 \norm{\cP_h \btau}^2_{\HdivO} = \norm{\cP_h \btau}^2_{\bcW(h)} \leq 
 2\norm{\btau}^2_{\bcW(h)}  + 2 \norm{\btau - \cP_h \btau}^2_{\bcW(h)} \leq 2(2 + C_0) \norm{\btau}^2_{\bcW(h)}.
\end{equation}
On the other hand, 
\[
 \norm{\btau}^2_{\bcW(h)} \leq 2\norm{\cP_h\btau}^2_{\HdivO} + 2\norm{\btau - \cP_h\btau}^2_{\bcW(h)} \leq 
 2(1 + C_0)\left(\norm{\cP_h\btau}^2_{\HdivO} + \norm{h_{\cF}^{-1/2} \jump{\btau}}^2_{0,\cF^*_h}\right) 
\]
for all $\btau \in \bcW_h$ which gives the upper bound of \eqref{equivN}. 
\end{proof}

\begin{prop}\label{infsupDh}
Assume that $\frac{1}{1+\kappa^2}\notin \sp(\bT)$. There exist parameters $h^*>0$  and $\textup{\texttt{a}}^*>0$ 
such that, for $h\leq h^*$ and $\textup{\texttt{a}}\ge \textup{\texttt{a}}^*$, 
\begin{equation}\label{infsupABh}
  \sup_{\taush\in \bcW_h\times \bcQ_h} \frac{\Dh{\sigmarh, \taush}}{\norm{\taush}_{DG}} 
  \geq \alpha_D^{d}  \norm{\sigmarh}_{DG}, \forall \sigmarh \in \bcW_h\times \bcQ_h
\end{equation}
with  $\alpha_D^{d}>0$ independent of the mesh size $h$ and $\lambda$. 
\end{prop}
\begin{proof}
We deduce from \eqref{infsupABh0} that there exists an operator 
$\Xi_h:\, \bcW_h^c\times \bcQ_h \to \bcW_h^c\times \bcQ_h$ 
such that, for $h\leq h_0$, 
\begin{equation}\label{Xi}
 \D{\taush, \Xi_h \taush} = \,  \norm{\taush}^2 \quad \text{and}\quad 
 \norm{\Xi_h \taush} \leq \frac{1}{\alpha_D^c} \norm{\taush}
\end{equation}
for all $\taush \in \bcW_h^c\times \bcQ_h$. 
Given $\taush\in \bcW_h\times \bcQ_h$, the decomposition $\btau_h = \btau_h^c + \tilde\btau_h$, with   
$\btau_h^c := \cP_h \btau_h$ and $\tilde\btau_h := \btau_h - \cP_h \btau_h$, and \eqref{Xi} yield   
\begin{multline}\label{split}
\Dh{\taush, \Xi_h \taushc+ \tausht} 
= \norm{\taushc}^2 +
\Dh{\taushc, \tausht} + \\
\Dh{ \tausht, \Xi_h \taushc} + 
\Dh{ \tausht, \tausht}.
\end{multline} 
Using the Cauchy-Schwarz inequality we have that  
\begin{multline*}
 \Dh{ \tausht, \tausht}= \norm{\bdiv_h \tilde \btau_h}_{0,\O}^2 + 
 \texttt{a} \norm{h_\cF^{-1/2} \jump{ \btau_h}}_{0,\cF^*_h}^2 - 
 \kappa^2 \int_{\O} \cC^{-1}\tilde \btau_h:\tilde \btau_h \\
 - 2 \int_{\cF^*_h} \mean{\bdiv_h \tilde \btau_h}\cdot \jump{\tilde \btau_h}
 \,\, \ge\,\, \texttt{a} \norm{h_\cF^{-1/2} \jump{ \btau_h}}_{0,\cF^*_h}^2 - C_1 \norm{ \tilde \btau_h }^2_{0,\O}\\
   -2 \norm{h_\cF^{1/2}  \mean{\bdiv_h \tilde \btau_h}}_{0,\cF^*_h} \norm{h_\cF^{-1/2} \jump{ \btau_h}}_{0,\cF^*_h}.
\end{multline*}
It follows from \eqref{discTrace} and \eqref{propB} that 
\[
 \Dh{ \tausht, \tausht} \geq (\texttt{a} -C_2)\norm{h_\cF^{-1/2} \jump{ \btau_h}}_{0,\cF^*_h}^2  
\]
with a constant $C_2>0$ independent of $h$ and $\lambda$. 

The remaining two terms of \eqref{split} are bounded from below by using \eqref{boundDh}, \eqref{discTrace} and \eqref{propB}. 
Indeed, it is straightforward that 
\begin{multline*}
 \Dh{\taushc, \tausht} \geq -M^{d}_D
  \norm{\taushc} \norm{\tilde\btau_h}_{\bcW(h)}\\
 \ge   -\frac{1}{4} \norm{\taushc}^2 - C_3  
 \norm{h_\cF^{-1/2} \jump{ \btau_h}}_{0,\cF^*_h}^2
\end{multline*}
and 
\begin{multline*}
 \Dh{ \tausht, \Xi_h \taushc} \geq - M^{d}_D 
 \norm{\Xi_h \taushc} \norm{\tilde\btau_h}_{\bcW(h)} \geq - \frac{M^{d}_D}{\alpha_D^{c}}   
 \norm{ \taushc} \norm{\tilde\btau_h}_{\bcW(h)} \\
 \geq -\frac{1}{4} \norm{  \taushc}^2 -  C_4 
 \norm{h_\cF^{-1/2} \jump{ \btau_h}}_{0,\cF_h}^2,
\end{multline*}
with $C_3>0$ and $C_4>0$ independent of $h$ and $\lambda$. Summing up, we have that,
\begin{equation*}
 \Dh{\taush, \Xi_h \taushc+ \tausht}\ge  \frac{1}{2} \norm{  \taushc}^2 + \big(\texttt{a}  - C^* \big) 
 \norm{h_\cF^{-1/2}\jump{ \btau_h}}_{0,\cF_h}^2,
\end{equation*}
with $C^* := C_2+C_3+C_4$. Hence, if $\texttt{a} > C^* + 1/2$, by virtue of \eqref{equivN} we have that 
\[
 \Dh{\taush, \Xi_h \taushc+ \tausht}\ge \frac{1}{2}  \Big( \norm{\taushc}^2 + 
 \norm{h_{\cF}^{-1/2} \jump{\btau}}^2_{0,\cF^*_h} \Big) \geq \frac{{\underbar C}^2}{2} \norm{ \taush }^2_{DG}.
\]
Finally, using \eqref{propB} and \eqref{Xi} we deduce that there exists $\alpha^d_D>0$ such that, 
\begin{equation*}
 \Dh{\taush, \Xi_h \taushc+ \tausht}
\geq \alpha^d_D 
\norm{\taush }_{DG} 
\Big(\norm{\Xi_h \taushc + \tausht}^2_{DG}\Big), 
\end{equation*}
provided that $h$ is sufficiently small and 
$\texttt{a}$ is sufficiently large,     
which gives \eqref{infsupABh}.
\end{proof}

The first consequence of the inf-sup condition \eqref{infsupABh} is that the DG problem \eqref{DGshort} 
admits a unique solution. Moreover,  we have the following C\'ea estimate.

\begin{theorem}
Assume that $\frac{1}{1+\kappa^2}\notin\sp(\bT)$ and let 
$\sigmar\in \bcW\times \bcQ$ be the solution of 
\eqref{varForm-a}--\eqref{varForm-b}. There exist parameters $h^*>0$  and $\textup{\texttt{a}}^*>0$ 
such that, for $h\leq h^*$ and $\textup{\texttt{a}}\ge \textup{\texttt{a}}^*$, 
 \begin{multline*}
 \norm{\sigmar - \sigmarh}_{DG} \leq (1+ \frac{M^{d}_D}{\alpha_D^{d}}) \inf_{\taush\in \bcW_h \times \bcQ_h} 
 \Big(
 \norm{\bsig - \btau_h}_{\bcW(h)} \\+ 
 \norm{h_{\cF}^{1/2} \mean{\bdiv (\bsig-\btau_h)}}_{0,\cF^*_h} + \norm{\br-\bs_h}_{0,\O}
 \Big).
\end{multline*}
Moreover, if the exact solution $\bu$ of \eqref{model} belongs to  $\H^{1+s}(\O)^3$   
for some $s>1/2$ and if $\bdiv \bsig \in \H^s(\O)^3$,  then the error estimate  
\[
 \norm{\sigmar - \sigmarh}_{DG} \leq \, C\,  h^{\min(s, k)}\,  \Big(
 \norm{\bu}_{1+s,K} + \norm{\bdiv \bsig}_{s,K}
 \Big), \quad \forall h\leq h^*,
\]
holds true with a constant  $C>0$ independent of $h$ and $\lambda$.
\end{theorem}
\begin{proof}
The first estimate follows from \eqref{boundDh}, \eqref{consistent} and \eqref{infsupABh}  as shown in 
\cite[Theorem 1.35]{DiPietroErn}. On the other hand, under the regularity hypotheses on $u$ and $\bsig$,
\begin{multline*}
 \norm{\sigmar - \sigmarh}_{DG} \leq (1+ \frac{M^{d}_D}{\alpha_D^{d}}) 
 \Big(
 \norm{\bsig - \Pi_h\bsig}_{\HdivO} + 
 \norm{h_{\cF}^{1/2} \mean{\bdiv (\bsig-\Pi_h\bsig)}}_{0,\cF^*_h}\\ + \norm{\br-\mathcal{S}_h\br}_{0,\O}
 \Big)
\end{multline*}
and we notice that 
\begin{equation*}
  \norm{h_{\cF}^{1/2} \mean{\bdiv (\bsig-\Pi_h\bsig)}}_{0,\cF^*_h}  
  \leq   
   \sum_{K\in \cT_h} \sum_{F\in \cF(K)} h_F\norm{ \bdiv (\bsig-\Pi_h\bsig) }^2_{0,F}.
\end{equation*}
Using the commuting diagram property satisfied by $\Pi_h$, the trace theorem and standard scaling arguments we 
obtain that 
\[
 h_F^{1/2}\norm{ \bdiv (\bsig-\Pi_h\bsig) }_{0,F} = h_F^{1/2}\norm{ \bdiv \bsig- \mathcal R_K \bdiv\bsig }_{0,F} 
 \leq C_2 h_K^{\min(k,s)} \norm{\bdiv \bsig}_{s,K}
\]
for all $F\in \cF(K)$, where the $\L^2(K)$-orthogonal projection $\mathcal R_K:= \mathcal R_h|_K$ 
onto $\cP_{k-1}(K)$ is applied componentwise. Consequently, by virtue of the error estimates 
\eqref{asymp}, \eqref{asympDiv} and \eqref{asymQ}, 
\[
 \norm{\sigmar - \sigmarh}_{DG} \leq C_3  h^{\min(k,s)}
 \Big(
 \norm{\bu}_{1+s,K} + \norm{\bdiv \bsig}_{s,K}
 \Big)
\]
and the result follows.
\end{proof}

%
%

\section{Numerical results}\label{section:6}
We present a series of numerical experiments confirming the good performance of the continuous Galerkin 
scheme \eqref{CGForm} and the discontinuous Galerkin scheme \eqref{DGshort}. 
For simplicity we consider our model problem in two dimensions. 
The corresponding theory and results from three dimensions apply with trivial modifications.

All the numerical results have been obtained by using the FEniCS 
Problem Solving Environment \cite{fenics}. 
We choose $\Omega=(0,1)\times (0,1)$, $\lambda=\mu=1$ and select the data $\bF$ so that the exact solution is given by 
\[
\bu(x_1, x_2) = \begin{pmatrix}
-x_2 \sin(\kappa \pi x_1)
\\ 0.5 \pi x_2 \cos(\kappa \pi x_1)                                        
\end{pmatrix}.
\]
We also assume that the body is fixed on the whole $\partial \O$ and the non-homogeneous 
Dirichlet boundary condition is imposed by adding an adequate boundary term to the right-hand side 
of \eqref{DGshort}. 
The numerical results obtained below for the continuous and 
discontinuous Galerkin schemes have been obtained by considering nested sequences of 
uniform triangular meshes $\cT_h$ of the unit square $\O$. 
The individual relative errors produced by the continuous 
Galerkin  method are given by 
\begin{equation}\label{Ec}
 \texttt{e}^\kappa_c(\bsig) := \frac{\| \bsig - \bsig_h \|_{\HdivO}}{\| \bsig \|_{\HdivO}} \qquad \text{and} \qquad 
\texttt{e}^\kappa_c(\br) := \frac{\| \br - \br_h \|_{0,\O}}{\| \br \|_{0,\O}},
\end{equation}
where $(\bsig, \br)\in \bcW\times \bcQ$ and $(\bsig_h, \br_h)\in \bcW_h^c\times \bcQ_h$ are the solutions 
of \eqref{varForm} and \eqref{CGForm} respectively. We introduce the experimental rates of convergence
\begin{equation}\label{EcO}
\texttt{r}_c^{\kappa}(\bsig) := \frac{\log (\texttt{e}_c^{\kappa}(\bsig)/ 
\hat{\texttt{e}}_c^{\kappa}(\bsig))}{\log (h/\hat{h})},\qquad 
\texttt{r}_c^{\kappa}(\br) := 
\frac{\log (\texttt{e}_c^{\kappa}(\br)/ \hat{\texttt{e}}_c^{\kappa}(\br))}{\log (h/\hat{h})},
\end{equation}
where $\texttt{e}_c^{\kappa}$ and $\hat{\texttt{e}}_c^{\kappa}$ are the errors corresponding to two
consecutive triangulations with mesh sizes $h$ and $\hat{h}$, respectively.

Similarly, we denote the individual relative errors of the 
discontinuous Galerkin scheme 
\begin{equation}\label{Ed}
 \texttt{e}^\kappa_{d}(\bsig) := \frac{\| \bsig - \bsig_h \|_{\bcW(h)}}{\| \bsig \|_{\HdivO}}, \qquad 
\texttt{e}^\kappa_{d}(\br) := \frac{\| \br - \br_h \|_{0,\O}}{\| \br \|_{0,\O}},
\end{equation}
where, in this case, $(\bsig_h,\br_h)\in \bcW_h\times \bcQ_h$ is the solution of \eqref{DGshort}.
Accordingly, the experimental rates of convergence of the DG scheme are given by 
\begin{equation}\label{EdO}
\texttt{r}_d^{\kappa}(\bsig) := \frac{\log (\texttt{e}_d^{\kappa}(\bsig)/ 
\hat{\texttt{e}}_d^{\kappa}(\bsig))}{\log (h/\hat{h})}\,,\qquad 
\texttt{r}_d^{\kappa}(\br) := 
\frac{\log (\texttt{e}_d^{\kappa}(\br)/ \hat{\texttt{e}}_d^{\kappa}(\br))}{\log (h/\hat{h})}.
\end{equation}

We begin by testing the convergence order of the continuous Galerkin method \eqref{CGForm} for the range of 
values $\kappa = 4, 8, 16, 32$ of the wave number. We 
report in Tables \ref{table:sCG2}, \ref{table:rCG2}, \ref{table:sCG3}, \ref{table:rCG3} 
the relative errors \eqref{Ec} and the convergence orders \eqref{EcO} obtained 
in the cases $k=2$ and $k= 4$, respectively. 
It is clear that the correct quadratic and quartic convergence rates of 
the errors are attained in each variable and for each fixed wave number $\kappa$. 

\begin{table}[!p]
\footnotesize
\singlespacing
\tabcolsep=0.12cm
\begin{center}
\begin{tabular}{c | cc | cc | cc | cc }
 $1/h$  & $\verb"e"_c^4(\boldsymbol{\sigma})$ & $\texttt{r}_c^4(\boldsymbol{\sigma})$ 
 & $\verb"e"_c^8(\boldsymbol{\sigma})$ & $\texttt{r}_c^8(\boldsymbol{\sigma})$ & $\verb"e"_c^{16}(\boldsymbol{\sigma})$ 
 & $\texttt{r}_c^{16}(\boldsymbol{\sigma})$ & $\verb"e"_c^{32}(\boldsymbol{\sigma})$ 
 & $\texttt{r}_c^{32}(\boldsymbol{\sigma})$ \cr
\hline

$8$ & 6.90e$-$02 & $-$ & 2.96e$-$01 & $-$ & 1.64e$+$00 & $-$ & 2.43e$+$00 & $-$ \cr

$16$ & 1.79e$-$02 & 1.94 & 6.80e$-$02 & 2.12 & 3.09e$-$01 & 2.41 & 1.11e$+$00 & 1.12 \cr

$32$ & 4.53e$-$03 & 1.99 & 1.77e$-$02 & 1.94 & 6.78e$-$02 & 2.19 & 3.01e$-$01 & 1.89 \cr

$64$ & 1.13e$-$03 & 2.00 & 4.46e$-$03 & 1.99 & 1.76e$-$02 & 1.94 & 6.77e$-$02 & 2.16 \cr

$128$ & 2.84e$-$04 & 2.00 & 1.12e$-$03 & 2.00 & 4.44e$-$03 & 1.99 & 1.76e$-$02 & 1.94 \cr

$256$ & 7.10e$-$05 & 2.00 & 2.80e$-$04 & 2.00 & 1.11e$-$03 & 2.00 & 4.44e$-$03 & 1.99 \cr

\end{tabular}
\caption{Convergence of the CG method in $\bsig$ for different wave numbers ($k=2$). }
\label{table:sCG2}
\end{center}
\end{table}

\begin{table}[!p]
\footnotesize
\singlespacing
\tabcolsep=0.12cm
\begin{center}
\begin{tabular}{c | cc | cc | cc | cc  }
 $1/h$  &  $\verb"e"_c^4(\br)$ & $\texttt{r}_c^4(\br)$ & $\verb"e"_c^8(\br)$ 
 & $\texttt{r}_c^8(\br)$ & $\verb"e"_c^{16}(\br)$ & $\texttt{r}_c^{16}(\br)$ 
 & $\verb"e"_c^{32}(\br)$ & $\texttt{r}_c^{32}(\br)$ \cr
\hline

$8$ & 1.96e$-$01 & $-$ & 7.64e$-$01 & $-$ & 9.97e$+$00 & $-$ & 1.59e$+$01 & $-$ \cr

$16$ & 5.32e$-$02 & 1.88 & 1.95e$-$01 & 1.98 & 9.32e$-$01 & 3.42 & 6.80e$+$00 & 1.23 \cr

$32$ & 1.36e$-$02 & 1.97 & 5.28e$-$02 & 1.88 & 1.95e$-$01 & 2.26 & 7.80e$-$01 & 3.12 \cr

$64$ & 3.43e$-$03 & 1.99 & 1.35e$-$02 & 1.97 & 5.27e$-$02 & 1.89 & 1.95e$-$01 & 2.00 \cr

$128$ & 8.60e$-$04 & 2.00 & 3.39e$-$03 & 1.99 & 1.34e$-$02 & 1.97 & 5.27e$-$02 & 1.89 \cr

$256$ & 2.15e$-$04 & 2.00 & 8.48e$-$04 & 2.00 & 3.38e$-$03 & 1.99 & 1.34e$-$02 & 1.97 \cr

\end{tabular}
\caption{Convergence of the CG method in  $\br$ for different wave numbers ($k=2$).}
\label{table:rCG2}
\end{center}
\end{table}

\begin{table}[!p]
\footnotesize
\singlespacing
\tabcolsep=0.12cm
\begin{center}
\begin{tabular}{c | cc | cc | cc | cc }
 $1/h$  &  $\verb"e"_c^4(\boldsymbol{\sigma})$ & $\texttt{r}_c^4(\boldsymbol{\sigma})$ 
 & $\verb"e"_c^8(\boldsymbol{\sigma})$ & $\texttt{r}_c^8(\boldsymbol{\sigma})$ & $\verb"e"_c^{16}(\boldsymbol{\sigma})$ 
 & $\texttt{r}_c^{16}(\boldsymbol{\sigma})$ & $\verb"e"_c^{32}(\boldsymbol{\sigma})$ 
 & $\texttt{r}_c^{32}(\boldsymbol{\sigma})$ \cr
\hline

$8$ & 9.33e$-$04 & $-$ & 1.67e$-$02 & $-$ & 2.34e$-$01 & $-$ & 1.50e$+$00 & $-$ \cr

$16$ & 5.95e$-$05 & 3.97 & 8.92e$-$04 & 4.23 & 1.70e$-$02 & 3.79 & 1.16e$-$01 & 3.69 \cr

$32$ & 3.74e$-$06 & 3.99 & 5.68e$-$05 & 3.97 & 8.81e$-$04 & 4.27 & 1.69e$-$02 & 2.78 \cr

$64$ & 2.34e$-$07 & 4.00 & 3.57e$-$06 & 3.99 & 5.61e$-$05 & 3.97 & 8.78e$-$04 & 4.27 \cr

$128$ & 1.46e$-$08 & 4.00 & 2.23e$-$07 & 4.00 & 3.53e$-$06 & 3.99 & 5.60e$-$05 & 3.97 \cr

\end{tabular}
\caption{Convergence of the CG method in  $\bsig$ for different wave numbers ($k=4$).}
\label{table:sCG3}
\end{center}
\end{table}

\begin{table}[!p]
\footnotesize
\singlespacing
\tabcolsep=0.12cm
\begin{center}
\begin{tabular}{c | cc | cc | cc | cc }
 $1/h$   & $\verb"e"_c^4(\br)$ & $\texttt{r}_c^4(\br)$ & $\verb"e"_c^8(\br)$ 
 & $\texttt{r}_c^8(\br)$ & $\verb"e"_c^{16}(\br)$ & $\texttt{r}_c^{16}(\br)$ 
 & $\verb"e"_c^{32}(\br)$ & $\texttt{r}_c^{32}(\br)$ \cr
\hline

$8$ & 1.88e$-$03 & $-$ & 3.31e$-$02 & $-$ & 1.52e$+$00 & $-$ & 8.98e$+$00 & $-$ \cr

$16$ & 1.23e$-$04 & 3.93 & 1.77e$-$03 & 4.23 & 3.64e$-$02 & 5.38 & 5.24e$-$01 & 4.10 \cr

$32$ & 7.80e$-$06 & 3.98 & 1.15e$-$04 & 3.94 & 1.75e$-$03 & 4.38 & 3.41e$-$02 & 3.94 \cr

$64$ & 4.90e$-$07 & 3.99 & 7.27e$-$06 & 3.98 & 1.13e$-$04 & 3.95 & 1.74e$-$03 & 4.29 \cr

$128$ & 3.27e$-$08 & 3.91 & 4.56e$-$07 & 4.00 & 7.15e$-$06 & 3.99 & 1.13e$-$04 & 3.95 \cr

\end{tabular}
\caption{Convergence of the CG method in  $\br$ for different wave numbers ($k=4$).}
\label{table:rCG3}
\end{center}
\end{table}

The subsequent numerical tests are for the discontinuous Galerkin scheme \eqref{DGshort}.
We present throughout Tables \ref{table:sDG4}, \ref{table:rDG4},  \ref{table:sDG6} and \ref{table:rDG6}
results corresponding to  $k=4$ with a range of 
wave numbers given by $\kappa = 4, 8, 16, 32$. We also show results corresponding to  $k=6$ 
with $\kappa = 16, 28, 32, 40$. For both polynomial degrees ($k=4, 6$) we take a stabilization parameter 
$\texttt{a}=100$. The expected rates of convergence are attained in all the cases. We notice that 
the higher  the value of the wave number $\kappa$ is, the smaller is the mesh size  needed to 
reduce the error below a given tolerance.

\begin{table}[!t]
\footnotesize
\singlespacing
\tabcolsep=0.12cm
\begin{center}
\begin{tabular}{c | cc | cc | cc | cc }
 $1/h$  &  $\verb"e"_d^4(\boldsymbol{\sigma})$ & $\texttt{r}_d^4(\boldsymbol{\sigma})$ 
 & $\verb"e"_d^8(\boldsymbol{\sigma})$ & $\texttt{r}_d^8(\boldsymbol{\sigma})$ 
 & $\verb"e"_d^{16}(\boldsymbol{\sigma})$ & $\texttt{r}_d^{16}(\boldsymbol{\sigma})$ 
 & $\verb"e"_d^{32}(\boldsymbol{\sigma})$ & $\texttt{r}_d^{32}(\boldsymbol{\sigma})$ \cr
\hline

$8$ & 9.41e$-$04 & $-$ & 1.68e$-$02 & $-$ & 2.35e$-$01 & $-$ & 1.49e$+$00 & $-$ \cr

$16$ & 6.01e$-$05 & 3.97 & 9.00e$-$04 & 4.22 & 1.70e$-$02 & 3.79 & 1.17e$-$01 & 3.68 \cr

$32$ & 3.78e$-$06 & 3.99 & 5.75e$-$05 & 3.97 & 8.89e$-$04 & 4.26 & 1.70e$-$02 & 2.78 \cr

$64$ & 2.36e$-$07 & 4.00 & 3.61e$-$06 & 3.99 & 5.68e$-$05 & 3.97 & 8.87e$-$04 & 4.26 \cr

$128$ & 1.48e$-$08 & 4.00 & 2.26e$-$07 & 4.00 & 3.57e$-$06 & 3.99 & 5.67e$-$05 & 3.97 \cr

\end{tabular}
\caption{Convergence of the DG method in  $\bsig$ for different wave numbers ($k=4$, $\texttt{a}=100$).}
\label{table:sDG4}
\end{center}
\end{table}

\begin{table}[!t]
\footnotesize
\singlespacing
\tabcolsep=0.12cm
\begin{center}
\begin{tabular}{c | cc | cc | cc | cc }
 $1/h$  &  $\verb"e"_d^4(\br)$ & $\texttt{r}_d^4(\br)$ & $\verb"e"_d^8(\br)$ 
 & $\texttt{r}_d^8(\br)$ & $\verb"e"_d^{16}(\br)$ & $\texttt{r}_d^{16}(\br)$ 
 & $\verb"e"_d^{32}(\br)$ & $\texttt{r}_d^{32}(\br)$ \cr
\hline

$8$ & 1.88e$-$03 & $-$ & 3.30e$-$02 & $-$ & 1.52e$+$00 & $-$ & 8.84e$+$00 & $-$ \cr

$16$ & 1.23e$-$04 & 3.93 & 1.77e$-$03 & 4.22 & 3.63e$-$02 & 5.38 & 5.22e$-$01 & 4.08 \cr

$32$ & 7.81e$-$06 & 3.98 & 1.15e$-$04 & 3.94 & 1.75e$-$03 & 4.38 & 3.40e$-$02 & 3.94 \cr

$64$ & 4.90e$-$07 & 3.99 & 7.27e$-$06 & 3.98 & 1.13e$-$04 & 3.95 & 1.74e$-$03 & 4.28 \cr

$128$ & 3.14e$-$08 & 3.97 & 4.56e$-$07 & 4.00 & 7.15e$-$06 & 3.99 & 1.13e$-$04 & 3.95 \cr

\end{tabular}
\caption{Convergence of the DG method in  $\br$ for different wave numbers ($k=4$, $\texttt{a}=100$).}
\label{table:rDG4}
\end{center}
\end{table}

\begin{table}[!t]
\footnotesize
\singlespacing
\tabcolsep=0.12cm
\begin{center}
\begin{tabular}{c | cc | cc | cc | cc  }
 $1/h$  & $\verb"e"_d^{16}(\boldsymbol{\sigma})$ & $\texttt{r}_d^{16}(\boldsymbol{\sigma})$ 
 & $\verb"e"_d^{28}(\boldsymbol{\sigma})$ & $\texttt{r}_d^{28}(\boldsymbol{\sigma})$ 
 & $\verb"e"_d^{32}(\boldsymbol{\sigma})$ & $\texttt{r}_d^{32}(\boldsymbol{\sigma})$ 
 & $\verb"e"_d^{40}(\boldsymbol{\sigma})$ & $\texttt{r}_d^{40}(\boldsymbol{\sigma})$ \cr
\hline

$8$ & 1.46e$-$02 & $-$ & 4.41e$-$01 & $-$ & 4.89e$-$01 & $-$ & 1.24e$+$00 & $-$ \cr

$16$ & 3.69e$-$04 & 5.30 & 8.02e$-$03 & 5.78 & 9.42e$-$03 & 5.70 & 5.25e$-$02 & 4.56 \cr

$32$ & 4.80e$-$06 & 6.27 & 1.30e$-$04 & 5.95 & 3.70e$-$04 & 4.67 & 1.05e$-$03 & 5.65 \cr

$40$ & 4.29e$-$07 & 5.96 & 1.19e$-$05 & 5.89 & 2.62e$-$05 & 6.53 & 9.74e$-$05 & 5.86 \cr

$56$ & 1.71e$-$07 & 5.97 & 4.77e$-$06 & 5.93 & 1.05e$-$05 & 5.91 & 3.93e$-$05 & 5.88 \cr

$64$ & 7.68e$-$08 & 5.98 & 2.16e$-$06 & 5.94 & 4.77e$-$06 & 5.93 & 1.79e$-$05 & 5.90 \cr

\end{tabular}
\caption{Convergence of the DG method in $\bsig$ for different wave numbers ($k=6$, $\texttt{a}=100$).}
\label{table:sDG6}
\end{center}
\end{table}

\begin{table}[!t]
\footnotesize
\singlespacing
\tabcolsep=0.12cm
\begin{center}
\begin{tabular}{c | cc | cc | cc | cc  }
 $1/h$  & $\verb"e"_d^{16}(\br)$ & $\texttt{r}_d^{16}(\br)$ & $\verb"e"_d^{28}(\br)$ 
 & $\texttt{r}_d^{28}(\br)$ & $\verb"e"_d^{32}(\br)$ & $\texttt{r}_d^{32}(\br)$ 
 & $\verb"e"_d^{40}(\br)$ & $\texttt{r}_d^{40}(\br)$ \cr
\hline

$8$ & 8.13e$-$03 & $-$ & 2.08e$+$00 & $-$ & 2.65e$+$00 & $-$ & 3.93e$+$00 & $-$ \cr

$16$ & 8.08e$-$04 & 6.65 & 2.49e$-$02 & 6.38 & 3.09e$-$02 & 6.42 & 1.67e$-$01 & 4.56 \cr

$32$ & 9.71e$-$06 & 6.38 & 2.62e$-$04 & 6.58 & 7.85e$-$04 & 5.30 & 2.09e$-$03 & 6.32 \cr

$40$ & 8.68e$-$07 & 5.96 & 2.40e$-$05 & 5.89 & 5.29e$-$05 & 6.65 & 1.96e$-$04 & 5.84 \cr

$56$ & 3.45e$-$07 & 5.97 & 9.64e$-$06 & 5.93 & 2.13e$-$05 & 5.91 & 7.94e$-$05 & 5.87 \cr

$64$ & 1.55e$-$07 & 5.99 & 4.36e$-$06 & 5.94 & 9.64e$-$06 & 5.93 & 3.61e$-$05 & 5.90 \cr

\end{tabular}
\caption{Convergence of the DG method in  $\br$ for different wave numbers ($k=6$, $\texttt{a}=100$).}
\label{table:rDG6}
\end{center}
\end{table}

To test the locking-free character of the method in the nearly incompressible case, 
we consider now Lam\'e coefficients $\lambda$ and $\mu$ corresponding to  
a Poisson ratio $\nu=0.499$ and a Young modulus $E = 10$. 
We fix the  polynomial degree to $k=2$, take a stabilization parameter $\texttt{a} = 50$ and 
report in Tables \ref{table:s_incompressible} and \ref{table:r_incompressible} 
the experimental rates of convergence for $\kappa=4,8,16,32$. We observe that 
the method is thoroughly robust for nearly incompressible materials. However, 
it seems that the pre-asymptotic region increases  in this case for big values of $\kappa$.

We now study the influence of $\kappa$ on the choice of the stabilization parameter 
$\texttt{a}$ of the discontinuous Galerkin scheme \eqref{DGshort}. To this end,  we present 
in Figure \ref{fig:a0k}  different approximations 
corresponding to $\kappa=2,4,8,16,32$, obtained with the mesh $h=1/32$ and a polynomial degree $k=3$. 
In each case, we represent in a double logarithmic scale the errors versus the parameter 
$\texttt{a}$. Clearly, $\texttt{a}$ is not sensible to the variations of $\kappa$. However, 
higher polynomial degrees $k$ require higher values for the stabilization parameter   
$\texttt{a}$. This is made clear in Figure \ref{fig:a0m} where the  polynomial degrees   
 $k=1,\cdots,7$ are considered on a fixed mesh $h=1/32$, with a fixed wave number  
$\kappa=16$. In each case,  the errors  are depicted versus  
$\texttt{a}$. 

\begin{table}[!t]
\footnotesize
\singlespacing
\tabcolsep=0.12cm
\begin{center}
\begin{tabular}{c | cc | cc | cc | cc }
 $1/h$  &  $\verb"e"_d^4(\boldsymbol{\sigma})$ & $\texttt{r}_d^4(\boldsymbol{\sigma})$ 
 & $\verb"e"_d^8(\boldsymbol{\sigma})$ & $\texttt{r}_d^8(\boldsymbol{\sigma})$ & $\verb"e"_d^{16}(\boldsymbol{\sigma})$ 
 & $\texttt{r}_d^{16}(\boldsymbol{\sigma})$ & $\verb"e"_d^{32}(\boldsymbol{\sigma})$ 
 & $\texttt{r}_d^{32}(\boldsymbol{\sigma})$ \cr
\hline

$8$ & 6.90e$-$02 & $-$ & 3.27e$-$01 & $-$ & 3.58e$-$02 & $-$ & 1.80e$-$02 & $-$ \cr

$16$ & 1.80e$-$02 & 1.94 & 6.83e$-$02 & 2.26 & 3.27e$-$01 & $-$ & 1.83e$-$02 & $-$ \cr

$32$ & 4.54e$-$03 & 1.99 & 1.78e$-$02 & 1.94 & 6.81e$-$02 & 2.27 & 3.27e$-$01 & $-$ \cr

$64$ & 1.14e$-$03 & 2.00 & 4.49e$-$03 & 1.99 & 1.77e$-$02 & 1.94 & 6.80e$-$02 & 2.27 \cr

$128$ & 2.84e$-$04 & 2.00 & 1.13e$-$03 & 2.00 & 4.48e$-$03 & 1.99 & 1.77e$-$02 & 1.94 \cr

\end{tabular}
\caption{Convergence of the DG method in $\bsig$ for different wave numbers ($k = 2$, $\texttt{a}=50$, $\nu = 0.499$).}
\label{table:s_incompressible}
\end{center}
\end{table}

\begin{table}[!t]
\footnotesize
\singlespacing
\tabcolsep=0.12cm
\begin{center}
\begin{tabular}{c | cc | cc | cc | cc }
 $1/h$  &  $\verb"e"_d^4(\br)$ & $\texttt{r}_d^4(\br)$ & $\verb"e"_d^8(\br)$ 
 & $\texttt{r}_d^8(\br)$ & $\verb"e"_d^{16}(\br)$ & $\texttt{r}_d^{16}(\br)$ 
 & $\verb"e"_d^{32}(\br)$ & $\texttt{r}_d^{32}(\br)$ \cr
\hline

$8$ & 1.00e$+$00 & $-$ & 1.15e$+$01 & $-$ & 5.19e$+$00 & $-$ & 3.60e$+$00 & $-$ \cr

$16$ & 1.38e$-$01 & 2.87 & 1.08e$-$00 & 3.42 & 1.26e$+$01 & $-$ & 2.92e$+$00 & $-$ \cr

$32$ & 2.08e$-$02 & 2.73 & 1.37e$-$01 & 2.98 & 1.13e$+$00 & 3.48 & 1.32e$+$01 & $-$ \cr

$64$ & 3.94e$-$03 & 2.40 & 2.04e$-$02 & 2.75 & 1.38e$-$01 & 3.03 & 1.16e$+$00 & 3.50 \cr

$128$ & 8.91e$-$04 & 2.15 & 3.88e$-$03 & 2.40 & 2.04e$-$02 & 2.76 & 1.39e$-$01 & 3.06 \cr

\end{tabular}
\caption{Convergence of the DG method in $\br$ for different wave numbers ($k = 2$, $\texttt{a}=50$, $\nu = 0.499$).}
\label{table:r_incompressible}
\end{center}
\end{table}

\begin{figure}[!t]
    \centering
    \subfloat[Convergence history for $\bsig$]{{\includegraphics[width=0.47\textwidth]{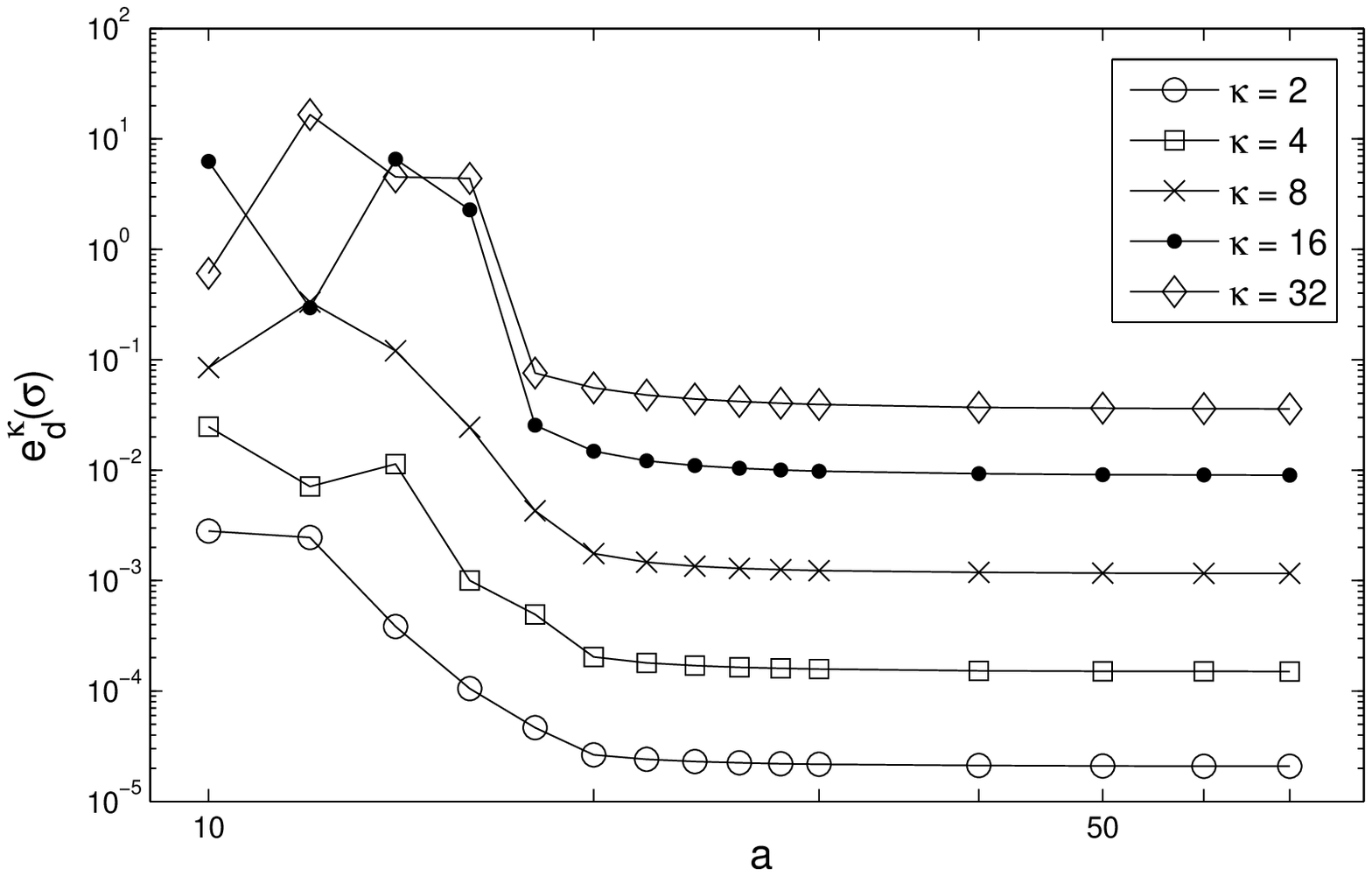}}}
    \qquad
    \subfloat[Convergence history for $\br$]{{\includegraphics[width=0.47\textwidth]{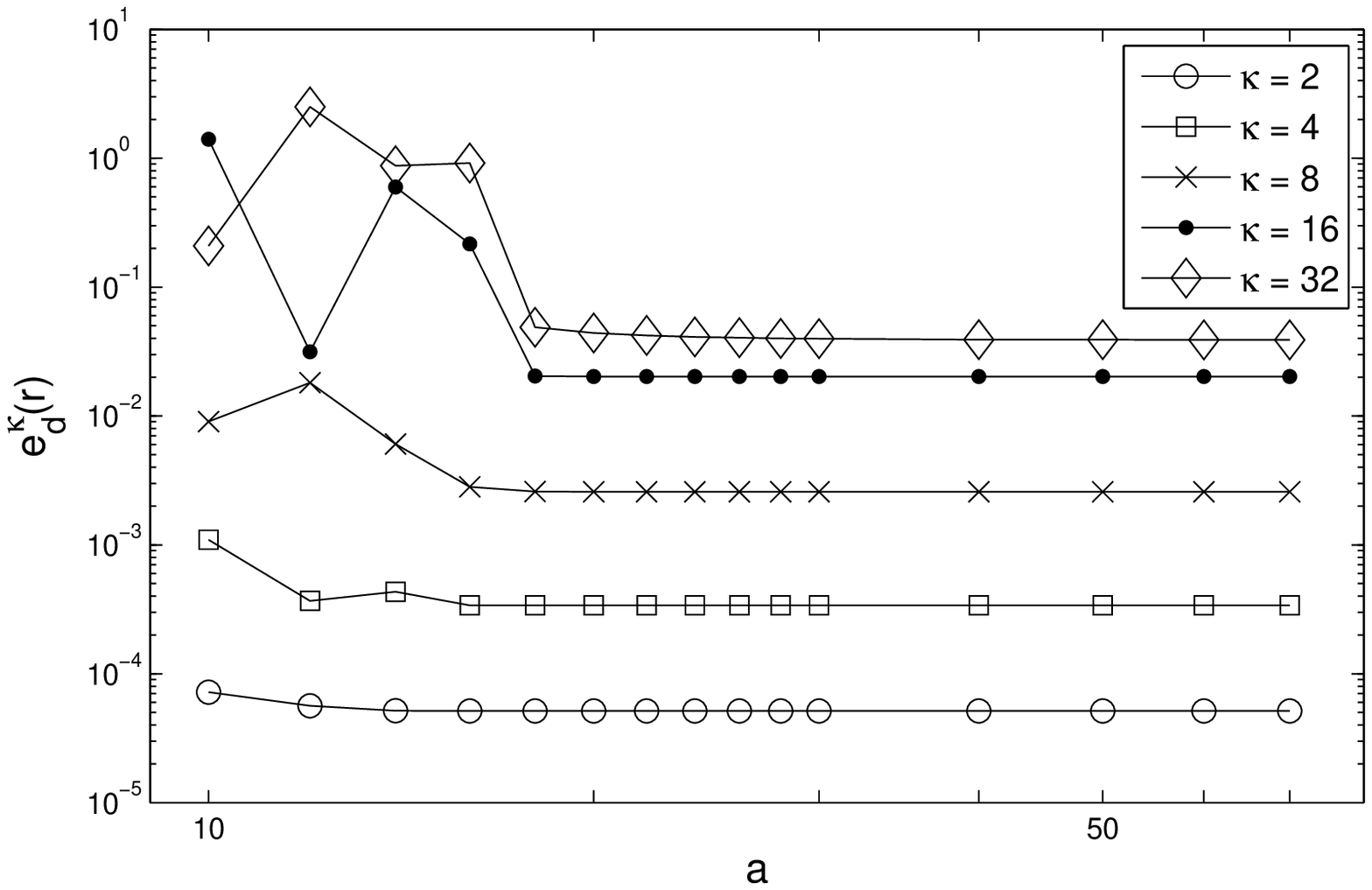}}}
    \caption{Errors of the DG method versus $\texttt{a}$  with $h= 1/32$ and $k= 3$.}%
    \label{fig:a0k}%
\end{figure}

\begin{figure}[!t]
    \centering
    \subfloat[Convergence history for $\bsig$]{{\includegraphics[width=0.47\textwidth]{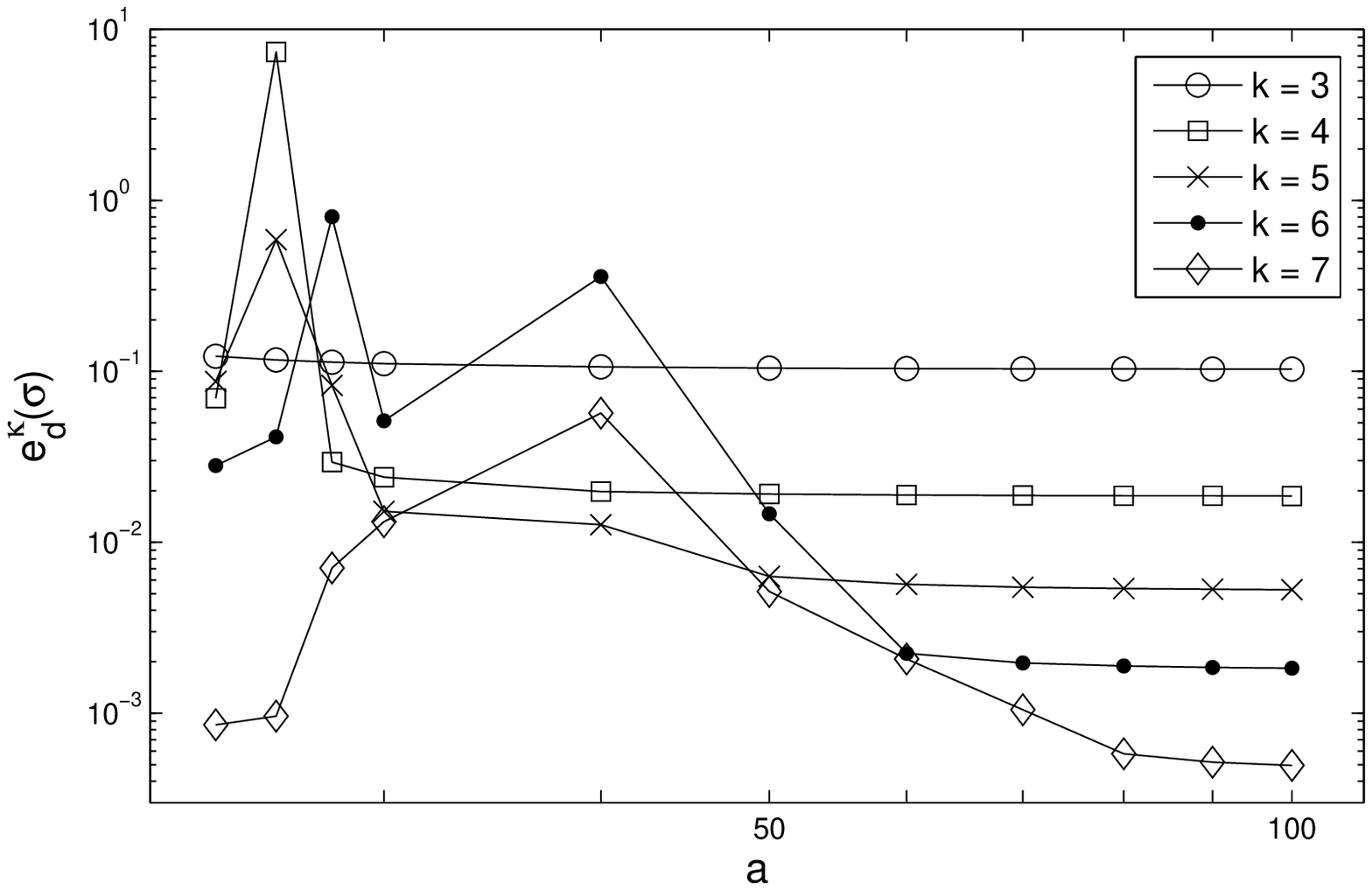}}}
    \qquad
    \subfloat[Convergence history for $\br$]{{\includegraphics[width=0.47\textwidth]{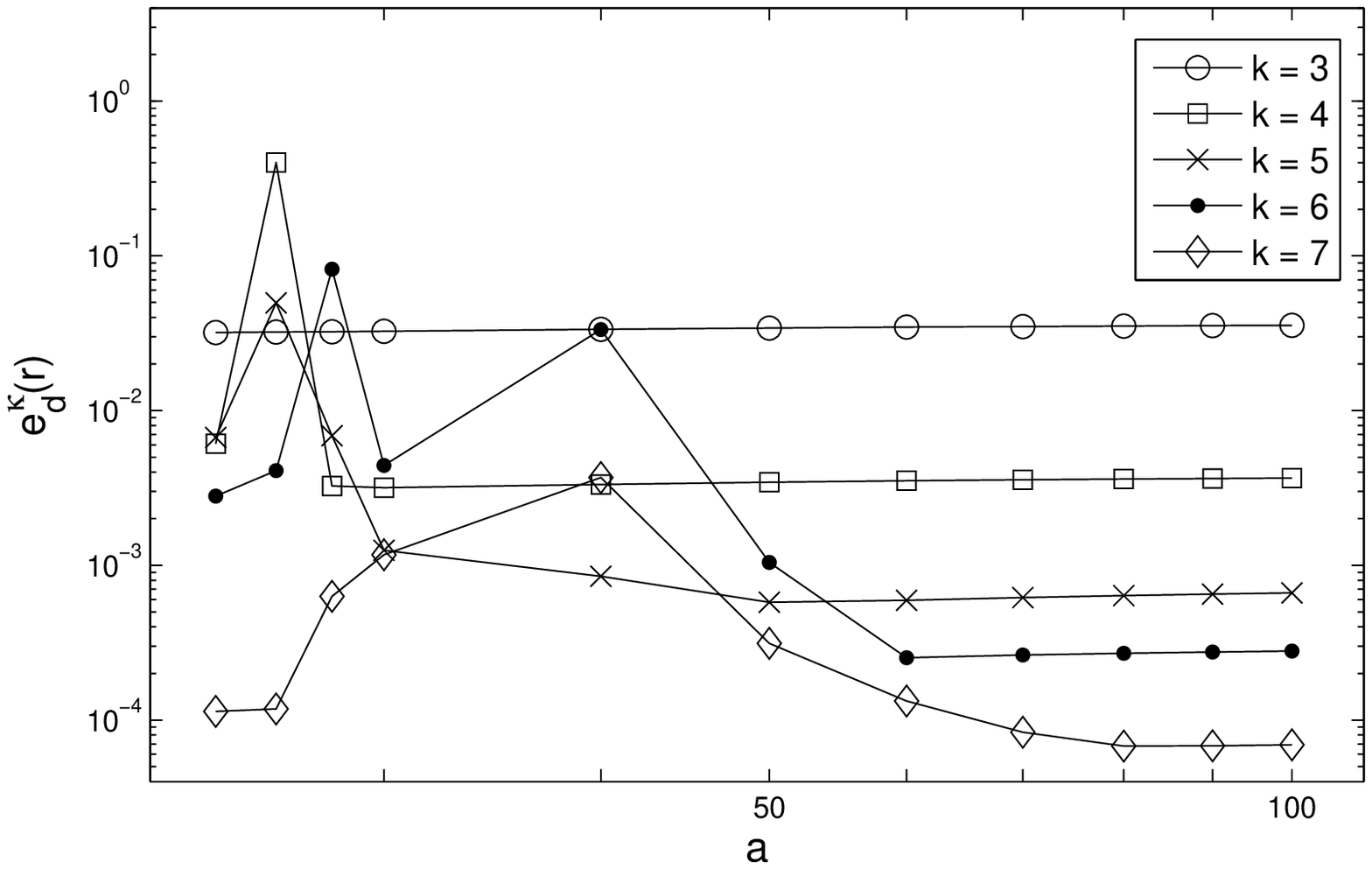}}}
    \caption{Errors of the DG method versus $\texttt{a}$  with $h=1/32$ and $\kappa$ = 16.}%
    \label{fig:a0m}%
\end{figure}


\end{document}